\documentclass[11pt]{amsart}
\usepackage{amsopn}
\usepackage{amssymb, amscd}

\newcommand{\nc}{\newcommand}

\nc{\fg}{\mathfrak{f} } \nc{\vg}{\mathfrak{v} } \nc{\wg}{\mathfrak{w} }
\nc{\zg}{\mathfrak{z} } \nc{\ngo}{\mathfrak{n} } \nc{\kg}{\mathfrak{k} }
\nc{\mg}{\mathfrak{m} } \nc{\bg}{\mathfrak{b} } \nc{\ggo}{\mathfrak{g} }
\nc{\ggob}{\overline{\mathfrak{g}} } \nc{\sog}{\mathfrak{so} }
\nc{\sug}{\mathfrak{su} } \nc{\spg}{\mathfrak{sp} } \nc{\slg}{\mathfrak{sl} }
\nc{\glg}{\mathfrak{gl} } \nc{\cg}{\mathfrak{c} } \nc{\rg}{\mathfrak{r} }
\nc{\hg}{\mathfrak{h} } \nc{\tg}{\mathfrak{t} } \nc{\ug}{\mathfrak{u} }
\nc{\dg}{\mathfrak{d} } \nc{\ag}{\mathfrak{a} } \nc{\pg}{\mathfrak{p} }
\nc{\sg}{\mathfrak{s} } \nc{\affg}{\mathfrak{aff} }

\nc{\pca}{\mathcal{P}} \nc{\nca}{\mathcal{N}} \nc{\lca}{\mathcal{L}}
\nc{\oca}{\mathcal{O}} \nc{\mca}{\mathcal{M}} \nc{\tca}{\mathcal{T}}
\nc{\aca}{\mathcal{A}} \nc{\cca}{\mathcal{C}} \nc{\gca}{\mathcal{G}}
\nc{\sca}{\mathcal{S}} \nc{\hca}{\mathcal{H}} \nc{\bca}{\mathcal{B}}
\nc{\dca}{\mathcal{D}} \nc{\val}{\operatorname{val}}

\nc{\vp}{\varphi} \nc{\ddt}{\frac{d}{dt}} \nc{\dds}{\frac{d}{ds}}
\nc{\dpar}{\frac{\partial}{\partial t}} \nc{\im}{\mathtt{i}}

\nc{\SO}{\mathrm{SO}} \nc{\Spe}{\mathrm{Sp}} \nc{\Sl}{\mathrm{SL}}
\nc{\SU}{\mathrm{SU}} \nc{\Or}{\mathrm{O}} \nc{\U}{\mathrm{U}} \nc{\Gl}{\mathrm{GL}}
\nc{\Se}{\mathrm{S}} \nc{\Cl}{\mathrm{Cl}} \nc{\Spein}{\mathrm{Spin}}
\nc{\Pin}{\mathrm{Pin}} \nc{\G}{\mathrm{GL}_n(\RR)} \nc{\g}{\mathfrak{gl}_n(\RR)}

\nc{\RR}{{\Bbb R}} \nc{\HH}{{\Bbb H}} \nc{\CC}{{\Bbb C}} \nc{\ZZ}{{\Bbb Z}}
\nc{\FF}{{\Bbb F}} \nc{\NN}{{\Bbb N}} \nc{\QQ}{{\Bbb Q}} \nc{\PP}{{\Bbb P}}

\nc{\vs}{\vspace{.2cm}} \nc{\vsp}{\vspace{1cm}} \nc{\ip}{\langle\cdot,\cdot\rangle}
\nc{\ipp}{(\cdot,\cdot)} \nc{\la}{\langle} \nc{\ra}{\rangle} \nc{\unm}{\tfrac{1}{2}}
\nc{\unc}{\tfrac{1}{4}} \nc{\und}{\tfrac{1}{16}} \nc{\no}{\vs\noindent}
\nc{\lam}{\Lambda^2(\RR^n)^*\otimes\RR^n} \nc{\tangz}{{\rm T}^{\rm Zar}}
\nc{\nor}{{\sf n}}  \nc{\mum}{/\!\!/} \nc{\kir}{/\!\!/\!\!/}
\nc{\Ri}{\tfrac{4\Ric_{\mu}}{||\mu||^2}} \nc{\ds}{\displaystyle}
\nc{\ben}{\begin{enumerate}} \nc{\een}{\end{enumerate}} \nc{\f}{\frac}
\nc{\lb}{[\cdot,\cdot]} \nc{\isn}{\tfrac{1}{||v||^2}}
\nc{\gkp}{(\ggo=\kg\oplus\pg,\ip)} \nc{\ukh}{(\ug=\kg\oplus\hg,\ip)}
\nc{\tgkp}{(\tilde{\ggo}=\kg\oplus\pg,\ip)}
\nc{\wt}{\widetilde} \nc{\mm}{M}

\nc{\Hess}{\operatorname{Hess}} \nc{\ad}{\operatorname{ad}}
\nc{\Ad}{\operatorname{Ad}} \nc{\rank}{\operatorname{rank}}
\nc{\Irr}{\operatorname{Irr}} \nc{\End}{\operatorname{End}}
\nc{\Aut}{\operatorname{Aut}} \nc{\Inn}{\operatorname{Inn}}
\nc{\Der}{\operatorname{Der}} \nc{\Ker}{\operatorname{Ker}}
\nc{\Iso}{\operatorname{I}} \nc{\Diff}{\operatorname{Diff}}
\nc{\Lie}{\operatorname{L}} \nc{\tr}{\operatorname{tr}} \nc{\dif}{\operatorname{d}}
\nc{\sen}{\operatorname{sen}} \nc{\modu}{\operatorname{mod}}
\nc{\CRic}{\operatorname{PP}} \nc{\Cric}{\operatorname{P}} \nc{\Ricci}{\operatorname{Ric}}
\nc{\sym}{\operatorname{sym}} \nc{\symac}{\operatorname{sym^{ac}}}
\nc{\symc}{\operatorname{sym^{c}}} \nc{\scalar}{\operatorname{sc}}
\nc{\grad}{\operatorname{grad}} \nc{\ricci}{\operatorname{Rc}}
\nc{\Nor}{\operatorname{Norm}}  \nc{\ricc}{\operatorname{Rc^{c}}} \nc{\Ricc}{\operatorname{Ric^{c}}} \nc{\ricac}{\operatorname{Rc^{ac}}} \nc{\Ricac}{\operatorname{Ric^{ac}}} \nc{\Riem}{\operatorname{Rm}}
\nc{\riccig}{\operatorname{ric^{\gamma}}} \nc{\Rin}{\operatorname{M}}
\nc{\Le}{\operatorname{L}} \nc{\tang}{\operatorname{T}}
\nc{\level}{\operatorname{level}} \nc{\rad}{\operatorname{r}}
\nc{\abel}{\operatorname{ab}} \nc{\CH}{\operatorname{CH}}
\nc{\mcc}{\operatorname{mcc}} \nc{\Adj}{\operatorname{Adj}}
\nc{\Order}{\operatorname{O}}  \nc{\inj}{\operatorname{inj}} \nc{\proy}{\operatorname{pr}}
\nc{\vol}{\operatorname{vol}} \nc{\Diag}{\operatorname{Diag}}

\theoremstyle{plain}
\newtheorem{theorem}{Theorem}[section]
\newtheorem{proposition}[theorem]{Proposition}
\newtheorem{corollary}[theorem]{Corollary}
\newtheorem{lemma}[theorem]{Lemma}

\theoremstyle{definition}
\newtheorem{definition}[theorem]{Definition}

\theoremstyle{remark}
\newtheorem{remark}[theorem]{Remark}

\newtheorem{example}[theorem]{Example}

\title[]{Curvature flows for almost-hermitian Lie groups}

\author{Jorge Lauret}

\address{Universidad Nacional de C\'ordoba, FaMAF and CIEM, 5000 C\'ordoba, Argentina}
\email{lauret@famaf.unc.edu.ar}

\thanks{This research was partially supported by grants from CONICET, FONCYT and SeCyT (Universidad Nacional de C\'ordoba)}

\begin{document}

\maketitle

\begin{abstract}
We study curvature flows in the locally homogeneous case (e.g. compact quotients of Lie groups, solvmanifolds, nilmanifolds) in a unified way, by considering a generic flow under just a few natural conditions on the broad class of almost-hermitian structures.  As a main tool, we use an ODE system defined on the variety of $2n$-dimensional Lie algebras, called the {\it bracket flow}, whose solutions differ from those to the original curvature flow by only pull-back by time-dependent diffeomorphisms.  The approach, which has already been used to study the Ricci flow on homogeneous manifolds, is useful to better visualize the possible pointed limits of solutions, under diverse rescalings, as well as to address regularity issues.  Immortal, ancient and self-similar solutions arise naturally from the qualitative analysis of the bracket flow.  The Chern-Ricci flow and the symplectic curvature flow are considered in more detail.
\end{abstract}



\section{Introduction}\label{intro}

The idea of evolving geometric structures to study them, or their underlying manifolds, is quite old.  However, it has been the resolution by Perelman of the Poincar\'e and Thurston Geometrization Conjectures by using Hamilton's Ricci flow which seems to have placed this approach among the most active topics in differential geometry in the last decade.

More recently, there have appeared in the literature many promising proposals to adapt the Ricci flow machinery to complex and symplectic geometry, all of which coincide with the K\"ahler-Ricci flow when the starting structure happens to be K\"ahler.  Among them, we have the following flows for hermitian metrics on a fixed complex manifold: the {\it hermitian curvature flow} (see \cite{StrTn1}), called {\it pluriclosed flow} in the SKT case (see \cite{StrTn4,StrTn3}), and the {\it Chern-Ricci flow} (see \cite{TstWnk}).  On the wider class of almost-hermitian manifolds, a large family of curvature flows associated to the Chern connection was studied in \cite{StrTn2}, which includes the gradient flow introduced
in \cite{Vzz1}.  The {\it symplectic curvature flow} introduced in \cite{StrTn2} evolves almost-K\"ahler manifolds, and coincides with the {\it anti-complexified Ricci flow} studied in \cite{LeWng} when the symplectic structure remains fixed (i.e. Chern-Ricci flat case).

All these flows are defined by evolution equations of the form
\begin{equation}\label{intro1}
\left\{\begin{array}{l}
\dpar\omega=-2p, \\ \\
\dpar g=-2q,
\end{array}\right.
\end{equation}
for a one-parameter family of almost-hermitian structures $(\omega(t),g(t),J(t))$ on a given differentiable manifold $M$, where $p=p(\omega,g)\in\Lambda^2M$ is a $2$-form and $q=q(\omega,g)\in\sca^2M$ is a symmetric $2$-tensor (both diffeomorphism invariant), which are curvature tensors associated to some hermitian connection of the structure $(\omega,g)$.  In this paper, we call equation \eqref{eq} the $(p,q)$-{\it flow}, in order to study them in a unified way in the (locally) homogeneous case.  The compatibility of the solution $(\omega(t),g(t))$ is equivalent to
$$
q^{1,1}=p^{1,1}(\cdot,J\cdot), \qquad\forall t,
$$
and it follows from the formula $\omega=g(J\cdot,\cdot)$ that the evolution of $J$ is given by
$$
\dpar J=-2J(P^{ac}+ Q^{ac}),
$$
where $P,Q\in\End(TM)$ are defined by $p=\omega(P\cdot,\cdot)$, $q=g( Q\cdot,\cdot)$ and $A^{ac}:=\unm(A+JAJ)$.

Starting with the pioneer article \cite{IsnJck}, the role of (locally) homogeneous manifolds in Ricci flow theory has been important, not only in inspiring some conjectures which ended up being true in the general case, but also in providing counterexamples to some others.  We expect that Lie groups may provide an even more useful framework in the evolution of almost-hermitian structures, due to the lack of explicit examples for many concepts and behaviors in complex and symplectic geometry.

Our aim in this paper is to study the $(p,q)$-flow evolution of compact almost-hermitian manifolds $(M,\omega,g)$ whose universal cover is a Lie group $G$ and such that if $\pi:G\longrightarrow M$ is the covering map, then $\pi^*\omega$ and $\pi^*g$ are left-invariant.  This is in particular the case of invariant structures on a quotient $M=G/\Gamma$, where $\Gamma$ is a cocompact discrete subgroup of $G$ (e.g. solvmanifolds and nilmanifolds).  A $(p,q)$-flow solution on $M$ is therefore obtained by pulling down the corresponding $(p,q)$-flow solution on the Lie group $G$, which by diffeomorphism invariance stays left-invariant.  Any $(p,q)$-flow therefore becomes an ODE for a pair $(\omega(t),g(t))$, where $\omega(t)$ is a non-degenerate $2$-form on the Lie algebra $\ggo$ of $G$, $g(t)$ is an inner product on $\ggo$, $p=p(\omega,g)\in\Lambda^2\ggo^*$ and  $q=q(\omega,g)\in\sca^2\ggo^*$.  Thus short-time existence (forward and backward) and uniqueness of the solutions are always guaranteed.

\subsection*{Bracket flow} Given a left-invariant almost-hermitian structure $(\omega_0,g_0)$ on a simply connected Lie group $G$, one has that
$$
(\omega,g)=h^*(\omega_0,g_0):=\left(\omega_0(h\cdot,h\cdot),g_0(h\cdot,h\cdot)\right),
$$
is also almost-hermitian for any $h\in\Gl(\ggo)$, and conversely, any almost-hermitian structure on $\ggo$ is of this form.  Moreover, the corresponding Lie group isomorphism
$$
\widetilde{h}:(G,\omega,g)\longrightarrow (G_\mu,\omega_0,g_0), \qquad\mbox{where}\qquad \mu=h\cdot\lb:=h[h^{-1}\cdot,h^{-1}\cdot],
$$
is an equivalence of almost-hermitian structures.  Here $\lb$ denotes the Lie bracket of the Lie algebra $\ggo$ and so $\mu$ defines a new Lie algebra (isomorphic to $(\ggo,\lb)$) with same underlying vector space $\ggo$.  We denote by $G_\mu$ the simply connected Lie group with
Lie algebra $(\ggo,\mu)$.  As in the case of Ricci flow (see \cite{nilricciflow,homRF}), this parametrization of left-invariant almost-hermitian structures as points in the variety $\lca_{2n}$ of $2n$-dimensional Lie algebras suggests the following natural question: how does a $(p,q)$-flow look on $\lca_{2n}$?

We consider for a family $\mu(t)\in \Lambda^2\ggo^*\otimes\ggo$ of brackets the following evolution equation:
\begin{equation}\label{intro2}
\ddt\mu=\delta_\mu(P_\mu+Q_\mu^{ac}), \qquad\mu(0)=\lb,
\end{equation}
where $P_\mu,Q_\mu\in\End(\ggo)$ are the curvature tensors corresponding to the almost-hermitian manifold $(G_\mu,\omega_0,g_0)$ and $\delta_\mu:\End(\ggo)\longrightarrow\Lambda^2\ggo^*\otimes\ggo$ is defined by
$$
\delta_\mu(A):=\mu(A\cdot,\cdot)+\mu(\cdot,A\cdot)-A\mu(\cdot,\cdot) = -\ddt|_{t=0} e^{tA}\cdot\mu, \qquad\forall A\in\End(\ggo).
$$
The variety $\lca_{2n}$ is invariant under equation \eqref{intro2}, which will be called the {\it $(p,q)$-bracket flow}, and our first result (see Section \ref{sec-BF}) shows that any $(p,q)$-flow is equivalent, in a precise way, to its corresponding $(p,q)$-bracket flow.

\begin{theorem}\label{intro3}
For a given simply connected almost-hermitian Lie group $(G,\omega_0,g_0)$ with Lie algebra $\ggo$, consider the families of almost-hermitian Lie groups
$$
(G,\omega(t),g(t)), \qquad (G_{\mu(t)},\omega_0,g_0),
$$
where $(\omega(t),g(t))$ is the solution to the $(p,q)$-flow starting at $(\omega_0,g_0)$ and $\mu(t)$ is the $(p,q)$-bracket flow solution starting at the Lie bracket $\lb$ of $\ggo$.  Then there exist Lie group isomorphisms $h(t):G\longrightarrow G_{\mu(t)}$ such that
$$
(\omega(t),g(t))=h(t)^*(\omega_0,g_0), \qquad\mbox{and}\qquad \mu(t)=h(t)\cdot\lb, \qquad\forall t.
$$
\end{theorem}

The following are direct consequences of the theorem:

\begin{itemize}
\item The maximal interval of time existence $(T_-,T_+)$ is the same for both flows, as it is also the behavior
of any kind of curvature, and so regularity issues can be addressed on the $(p,q)$-bracket flow.

\item Assume that a normalization $c_k\mu(t_k)\to\lambda$, as $t_k\to T_\pm$, and denote by $\vp_k:G\longrightarrow G_{c_k\mu(t_k)}$ the isomorphism with derivative $\tfrac{1}{c_k}h(t_k)$.  It follows from \cite[Corollary 6.20]{spacehm} that, after possibly passing to a subsequence, the almost-hermitian manifolds $\left(G,\tfrac{1}{c_k^2}\omega(t_k),\tfrac{1}{c_k^2}g(t_k)\right)$ converge in the pointed (or Cheeger-Gromov) sense to $(G_\lambda,\omega_0,g_0)$, as $k\to\infty$.  We note that $G_\lambda$ may be non-isomorphic, and even non-homeomorphic, to $G$.
\end{itemize}

\subsection*{Regularity} As a first application of the above theorem, we obtain the following general regularity result (see Section \ref{sec-reg}).

\begin{theorem}\label{intro4}
If a left-invariant $(p,q)$-flow solution $(\omega(t),g(t))$ on a Lie group has a finite-time singularity at $T_+$ (resp. $T_-$), then
$$
\int_0^{T_+} |P+Q^{ac}|\; dt =\infty \qquad \left(\mbox{resp.}\quad \int_{T_-}^0 |P+Q^{ac}|\; dt =\infty\right).
$$
\end{theorem}

This was proved for the pluriclosed flow on any compact manifold in \cite[Theorem 1.2]{StrTn3} and may be considered as the analogous to the result
for the Ricci flow given in \cite{Ssm}.  We also compute the evolution of the Ricci and scalar curvatures of $g(t)$ along a $(p,q)$-solution $(\omega(t),g(t))$, as well as the evolution of the norm of the Lie bracket $|\mu(t)|$.  Note that if $|\mu(t)|$ were non-increasing, then the long-time existence of the solution (i.e. $T_+=\infty$) would follow from a standard ODE argument.  This is precisely the technique applied in \cite{EnrFnVzz} to prove long-time existence for any pluriclosed solution on a nilmanifold.  An alternative way to prove long-time existence on a solvmanifold is by showing that the scalar curvature $R$ must blow up at a finite-time singularity, as it is well-known that $R\leq 0$ for any left-invariant metric on a solvable Lie group (see \cite{Lfn} for an application of this argument in the Ricci flow case).

\subsection*{Solitons} In the general case, an almost-hermitian manifold $(M,\omega,g)$ will flow self-similarly along the $(p,q)$-flow, in the sense that
$$
(\omega(t),g(t))=(c(t)\vp(t)^*\omega_0,c(t)\vp(t)^*g_0),
$$
for some $c(t)>0$ and $\vp(t)\in\Diff(M)$, if and only if
$$
\left\{
\begin{array}{l}
p(\omega,g)=c\omega+\lca_{X}\omega, \\ \\
q(\omega,g)=cg+\lca_{X}g,
\end{array}\right.
$$
for some $c\in\RR$ and a complete vector field $X$ on $M$.  In analogy to the terminology used in Ricci flow theory, we call such $(\omega,g)$ a {\it soliton} almost-hermitian structure.  On Lie groups, it is natural to consider a $(p,q)$-flow solution to be self-similar if the diffeomorphisms $\vp(t)$ above are actually Lie group automorphisms.  We prove in Section \ref{sec-self} that this is equivalent to the following condition.

\begin{definition}
An almost-hermitian structure $(\omega,g)$ on a Lie algebra $\ggo$ is called a $(p,q)$-{\it soliton} if for some $c\in\RR$ and $D\in\Der(\ggo)$,
$$
\left\{
\begin{array}{l}
P(\omega,g)=cI+\unm(D-JD^tJ), \\ \\
 Q(\omega,g)=cI+\unm(D+D^t).
\end{array}\right.
$$
\end{definition}

\begin{remark}
This notion is stronger than the general soliton condition above (see Remark \ref{soliton}).  Nevertheless, if a homogeneous soliton almost-hermitian structure $(M,\omega,g)$ is presented as $M=G/K$ for the full symmetry group $G$, then one can prove in much the same way as in \cite[Theorem 3.1]{Jbl3} that there exists a one-parameter family of equivariant diffeomorphisms $\phi_t\in\Aut(G/K)$ (i.e. automorphisms of $G$ taking $K$ onto $K$) such that $(\omega(t),g(t))=(c(t)\phi_t^*\omega,c(t)\phi_t^*g)$ is a solution to the $(p,q)$-flow starting at $(\omega,g)$.
\end{remark}

We also show that the simpler condition,
\begin{equation}\label{intro5}
P+Q^{ac}=cI+D,
\end{equation}
suggested by the relationship between a $(p,q)$-flow and its $(p,q)$-bracket flow given in Theorem \ref{intro3}, is actually enough to get a soliton.  Moreover, the $(p,q)$-bracket flow solution starting at a $(p,q)$-soliton for which \eqref{intro5} holds is simply given by $\mu(t)=(-2ct+1)^{-1/2}\lb$ and hence they are precisely the fixed points and only possible limits, backward and forward, of any normalized $(p,q)$-bracket flow solution $c(t)\mu(t)$. The absence of certain chaotic behavior for the $(p,q)$-bracket flow would imply that any $(p,q)$-soliton is actually of this kind (see Section \ref{BFsoliton}).

\subsection*{Chern-Ricci flow} In Section \ref{sec-CRF}, we study the Chern-Ricci flow $\dpar \omega=-2p$ for hermitian metrics on a fixed complex manifold $(M,J)$ (CRF for short), where $p$ is the Chern-Ricci form.  In the case of Lie groups, $p$ depends only on $J$, and so along a solution, $p(t)\equiv p_0$.  This implies that the CRF-solution starting at $(\omega_0,g_0)$ is simply given by
$$
\omega(t)=\omega_0-2tp_0,
$$
which therefore exists as long as the hermitian map $I-2tP_0$ is positive.  It follows that the integral of the Chern scalar curvature $\tr{P}$ must blow up at a finite-time singularity and that $\tr{P(t)}\leq \frac{C}{T_+-t}$, for some constant $C>0$.  By using Theorem \ref{intro3}, one obtains that the solution to the Chern-Ricci bracket flow $\ddt\mu=\delta_\mu(P_\mu)$ is given by
$$
\mu(t)=(I-2tP_0)^{1/2}\cdot\lb,
$$
and as an application, we give a structural result for Chern-Ricci solitons.  Concerning convergence, we prove that the normalized Chern-Ricci bracket flow $\mu(t)/|\mu(t)|$ always converges, as $t\to T_\pm$, to a nonabelian Lie bracket $\lambda_\pm$ such that $(G_{\lambda_\pm},\omega_0,g_0)$ is a Chern-Ricci soliton.  A more detailed study on how is the limit $\lambda_\pm$ related to the starting point is given in \cite{CRF}, where it is also addressed the existence problem for Chern-Ricci solitons on $4$-dimensional solvable Lie groups.

\subsection*{Symplectic curvature flow} Another $(p,q)$-flow we consider in particular is the symplectic curvature flow (SCF for short) for a one-parameter family $(\omega(t),g(t))$ of almost-K\"ahler structures introduced in \cite{StrTn2}:
\begin{equation}
\left\{\begin{array}{l}
\dpar\omega=-2p, \\ \\
\dpar g=-2p^{1,1}(\cdot,J\cdot)-2\ricci^{(2,0)+(0,2)},
\end{array}\right.
\end{equation}
where $p$ is the Chern-Ricci form of $(\omega,g)$ and $\ricci$ is the Ricci tensor of $g$ (see Section \ref{sec-SCF}).  It is conjectured in \cite{StrTn2} that a SCF-solution $(\omega(t),g(t))$ with $M$ compact exists smoothly as long as the cohomology class $[\omega(t)]\in H^2(M,\RR)$ belongs to the cone $\cca$ of all classes which can be represented by a symplectic form, which evolves by
$$
[\omega(t)]=-4t\pi c_1(M,\omega_0)+[\omega_0].
$$
In the case that $c_1(M,\omega_0)=0$, which in particular holds for invariant almost-K\"ahler structures on compact quotients $M=G/\Gamma$ of Lie groups, the solution is therefore expected to be immortal, i.e. $T_+=\infty$.  We confirm the conjecture for the anti-complexified Ricci flow on any compact solvmanifold and also for some explicit examples with $p\ne 0$ we give in detail for SCF, including some nilmanifolds already studied in \cite{Pk}.  Formula \eqref{intro5} has shown to be very useful for SCF.  In \cite{SCFmuA}, it has been found a SCF-soliton on most of symplectic $4$-dimensional Lie groups, and in \cite{Frn}, on most $2$ and $3$-step symplectic nilpotent Lie groups of dimension $6$.

\vs \noindent {\it Acknowledgements.} The author is grateful to Luigi Vezzoni and Cynthia Will for several very helpful conversations, and to the referee for many useful comments which improved the first version of the present paper.

\section{Some notation}\label{sec-not}

Let $\ggo$ be a real vector space.  The following notation will be used for $\ggo$ the tangent space $T_pM$ at a point of a differentiable manifold, as well as for the underlying vector space of a Lie algebra.  We consider an almost-hermitian structure $(\omega,g,J)$ on $\ggo$, that is, a $2$-form $\omega$ and an inner product $g$ such that if
$$
\omega=g(J\cdot,\cdot),
$$
then $J^2=-I$.  The above formula is therefore equivalent to $g=\omega(\cdot,J\cdot)$.

The transposes of a linear map $A:\ggo\longrightarrow\ggo$ with respect to $g$ and $\omega$ are respectively given by
$$
g(A\cdot,\cdot)=g(\cdot,A^t\cdot), \qquad \omega(A\cdot,\cdot)=\omega(\cdot,A^{t_\omega}\cdot), \qquad A^{t_\omega}=-JA^tJ,
$$
and if $p:\ggo\times\ggo\longrightarrow\RR$ is a bilinear map, then the complexified (or $J$-invariant) and anti-complexified (or anti-$J$-invariant) components are defined by
$$
A=A^c+A^{ac}, \qquad A^c:=\unm(A-JAJ), \qquad A^{ac}:=\unm(A+JAJ),
$$
and $p=p^c+p^{ac}$, where
$$
p^c=p^{(1,1)}:=\unm(p(\cdot,\cdot)+p(J\cdot,J\cdot)),  \qquad p^{ac}=p^{(2,0)+(0,2)}:=\unm(p(\cdot,\cdot)-p(J\cdot,J\cdot)).
$$
Consider also a $\ggo$-valued bilinear map $\mu:\ggo\times\ggo\longrightarrow\ggo$.  The natural `change of basis' actions of $\Gl(\ggo)$ are defined by
$$
h\cdot A:=hAh^{-1}, \qquad h\cdot p:=p(h^{-1}\cdot,h^{-1}\cdot), \qquad h\cdot\mu:=h\mu(h^{-1}\cdot,h^{-1}\cdot),\qquad \forall h\in\Gl(\ggo),
$$
and the corresponding representation $\pi:\End(\ggo)\longrightarrow\End((\ggo^*)^2\otimes\ggo)$ is given by
\begin{equation}\label{defpi}
\pi(H)\mu:=H\mu(\cdot,\cdot)-\mu(H\cdot,\cdot)-\mu(\cdot,H\cdot), \qquad\forall H\in\End(\ggo).
\end{equation}

\section{Curvature flows for almost-hermitian manifolds}\label{sec-CF}

Let $M$ be a differentiable manifold of dimension $2n$.  We consider for a one-parameter family of almost-hermitian structures $(\omega(t),g(t),J(t))$ on $M$ an evolution equation of the form
\begin{equation}\label{eq}
\left\{\begin{array}{l}
\dpar\omega=-2p, \\ \\
\dpar g=-2q,
\end{array}\right.
\end{equation}
where $p=p(\omega,g)\in\Lambda^2M$ is a $2$-form and $q=q(\omega,g)\in\sca^2M$ is a symmetric $2$-tensor, which will usually be some kind of curvature tensors associated to each structure $(\omega,g)$.  Equation \eqref{eq} will be called the $(p,q)$-{\it flow}.

It is natural to assume that $p$ and $q$ are invariant by diffeomorphisms:
$$
p(\vp^*\omega,\vp^*g)=\vp^*p(\omega,g), \qquad q(\vp^*\omega,\vp^*g)=\vp^*q(\omega,g), \qquad\forall\vp\in\Diff(M).
$$

\begin{example}\label{ex1}
In\cite{StrTn2}, short-time existence, uniqueness and some regularity results have been proved in the compact case for a large family of $(p,q)$-flow equations associated to the Chern connection.  For any of such flows, if $J_0$ is integrable then $J(t)\equiv J_0$ and $g(t)$ is a solution to the {\it hermitian curvature flow} studied in \cite{StrTn1}.  If in addition $g_0$ is K\"ahler (i.e. $\omega_0$ closed), then $g(t)$ is the K\"ahler-Ricci flow solution.  In the presence of an SKT structure, this evolution equation is called the {\it pluriclosed flow} and $p$ becomes the $(1,1)$-part of the Bismut-Ricci form (see \cite{StrTn4,StrTn3,EnrFnVzz}).  The hermitian curvature flow for almost-hermitian manifolds introduced in \cite{Vzz1} also belongs to the family of $(p,q)$-flows studied in \cite{StrTn2}.
\end{example}

\begin{example}\label{ex2}
The {\it Chern-Ricci flow} is another evolution equation for hermitian manifolds (see Section \ref{sec-CRF}).  Here $p$ is the Chern-Ricci form and $J(t)\equiv J_0$, i.e. the complex manifold is fixed.  Existence and uniqueness of the solutions, as well as some regularity and convergence results have been obtained in \cite{TstWnk}.  This flow also coincides with the K\"ahler-Ricci flow when the starting hermitian structure is K\"ahler.
\end{example}

\begin{example}\label{ex3}
In the almost-K\"ahler case (i.e. $\omega$ closed), by taking $p$ the Chern-Ricci form of $(\omega,g)$ and $q=p^c(\cdot,J\cdot)+\ricci^{ac}$, where $\ricci$ is the Ricci tensor of $g$, equation \eqref{eq} becomes the {\it symplectic curvature flow} introduced in \cite{StrTn2} (see Section \ref{sec-SCF}).  In the case when $p(\omega,g)=0$ for all time $t$, the symplectic structure $\omega$ is fixed and $g(t)$ solves the {\it anti-complexified Ricci flow} studied in \cite{LeWng}.
\end{example}

Let us now analyze under what conditions the evolution preserves compatibility.  If we take the corresponding operator type tensors $P,Q\in\End(TM)$ defined by
\begin{equation}\label{defPQ}
p=\omega(P\cdot,\cdot)= g(JP\cdot,\cdot), \qquad q=g( Q\cdot,\cdot),
\end{equation}
then $P^{t_\omega}=P$ (i.e. $P^t=-JPJ$) and $Q^t=Q$.  It follows from the formula $\omega=g(J\cdot,\cdot)$ that the evolution of $J$ is given by
\begin{equation}\label{eqJ}
\dpar J=-2R, \qquad \mbox{where}\qquad R=JP- QJ.
\end{equation}
This implies that, provided the starting point $(\omega_0,g_0)$ is compatible, the compatibility condition $J^2=-I$ holds for all time $t$ if and only if $RJ+JR=0$.  It follows that the compatibility of the solution $(\omega,g)$ for all time $t$ is equivalent to
\begin{equation}\label{comp}
    P^c=Q^c, \qquad\mbox{or equivalently}, \qquad q^c=p^c(\cdot,J\cdot), \qquad\forall t.
\end{equation}
(Compare with \cite[Lemma 4.2]{StrTn2}).  Indeed, if $P^c=Q^c$ then $\dpar J^2=-4(P-Q)(J^2+I)$ and so $J(t)^2\equiv-I$ by uniqueness of the solution starting at $J_0^2=-I$.  The converse follows from $0=RJ+JR=2(-P^c+Q^c)$.

It is therefore natural to assume that the pair $(p,q)$ of curvature tensors defining a $(p,q)$-flow evolution satisfies condition \eqref{comp} for any almost-hermitian structure.

By using \eqref{comp}, the evolution of the almost-complex structure can be rewritten as
\begin{equation}\label{eqJ2}
\dpar J=-2J(P^{ac}+ Q^{ac}).
\end{equation}
In the almost-K\"ahler case, we note that to have $p$ closed is clearly enough to get $\omega(t)$ closed for all $t$ provided $\omega_0$ is closed, by using that $\dpar d\omega=-2d p$.

\section{Curvature flows for Lie groups}\label{sec-LG}

Our aim in this paper is to study the $(p,q)$-flow evolution of compact almost-hermitian manifolds $(M,\omega,g)$ whose universal cover is a Lie group $G$ and such that if $\pi:G\longrightarrow M$ is the covering map, then $\pi^*\omega$ and $\pi^*g$ are left-invariant.  This is in particular the case of invariant structures on a quotient $M=G/\Gamma$, where $\Gamma$ is a cocompact discrete subgroup of $G$ (e.g. solvmanifolds and nilmanifolds).  A $(p,q)$-flow solution on $M$ is therefore obtained by pulling down the corresponding $(p,q)$-flow solution on the Lie group $G$, which by diffeomorphism invariance stays left-invariant and so it can be studied on the Lie algebra as an ODE.

Any almost-hermitian structure on a Lie group with Lie algebra $\ggo$ which is left-invariant is determined by a compatible pair $(\omega,g)$, where $\omega$ is a non-degenerate $2$-form on the vector space $\ggo$ and $g$ is an inner product on $\ggo$.  Any $(p,q)$-flow \eqref{eq} on $M$ or on the covering Lie group $G$ therefore becomes an ODE system of the form
\begin{equation}\label{eqLG}
\left\{\begin{array}{l}
\ddt\omega=-2p, \\ \\
\ddt g=-2q,
\end{array}\right.
\end{equation}
where $p=p(\omega,g)\in\Lambda^2\ggo^*$, $q=q(\omega,g)\in\sca^2\ggo^*$, since all the tensors involved are determined by their value at the identity of the group.  Thus short-time existence (forward and backward) and uniqueness (among left-invariant ones) of the solutions are always guaranteed.

We may consider a one-parameter family $(\omega(t),g(t))$ starting at $(\omega_0,g_0)$ in terms of a pair of operators $(\Omega(t), G(t))$ in $\Gl(\ggo)$ as follows,
$$
\left\{\begin{array}{l}
\omega(t)=\omega_0(\Omega(t)\cdot,\cdot), \qquad\Omega(0)=I, \\ \\
g(t)=g_0(G(t)\cdot,\cdot), \qquad G(0)=I,
\end{array}\right.
$$
and thus the $(p,q)$-flow defined as in \eqref{eqLG} for $(\omega,g)$ is equivalent to the ODE system
\begin{equation}\label{eqLG2}
\left\{\begin{array}{l}
\ddt\Omega=-2\Omega P,  \\ \\
\ddt G=-2G Q,
\end{array}\right.
\end{equation}
where $P,Q\in\End(\ggo)$ are defined as in \eqref{defPQ}.  It is easy to see that for all $t$,
\begin{equation}\label{jota}
    J_0\Omega=GJ.
\end{equation}
As there is only one of these compatible pairs $(\omega,g)$ on the vector space $\ggo$, up to the action of $\Gl(\ggo)$, we have that a solution to a $(p,q)$-flow \eqref{eqLG} starting at $(\omega_0,g_0)$ can always be written as
$$
(\omega(t),g(t))=(h^{-1}\cdot\omega_0,h^{-1}\cdot g_0) =\left(\omega_0(h\cdot,h\cdot),g_0(h\cdot,h\cdot)\right),
$$
for some $h=h(t)\in\Gl(\ggo)$, or equivalently,
\begin{equation}\label{OmGh}
\Omega(t)=h^{t_{\omega_0}}h=-J_0h^tJ_0h, \qquad G(t)=h^th,
\end{equation}
where $A^t$ will denote from now on the transpose of an operator $A$ with respect to $g_0$.  It follows that
\begin{equation}\label{Jh}
J(t)=h^{-1}J_0h, \qquad\forall t.
\end{equation}
We note that $h(t)$ is unique only up to left-multiplication by the unitary group
$$
\U(n):=\Spe(\omega_0)\cap\Or(g_0) =\{\vp\in\Gl(\ggo):\omega_0=\omega_0(\vp\cdot,\vp\cdot), \quad g_0=g_0(\vp\cdot,\vp\cdot)\},
$$
and that $h(t)$ can be taken to be differentiable on $t$ (recall that $\dim{\ggo}=2n$).

A good understanding of the evolution of such $h(t)\in\Gl(\ggo)$ may provide a useful tool as the whole $(p,q)$-flow solution is determined by $h(t)$.

\begin{lemma}\label{h}
Let $(\omega(t),g(t))$ be a $(p,q)$-flow solution starting at $(\omega_0,g_0)$.  If $h=h(t)\in\Gl(\ggo)$ is the solution to the ODE
$$
\ddt h=-h(P+Q^{ac}) =-h(P^{ac}+Q), \qquad h(0)=I,
$$
then $(\omega,g)=(h^{-1}\cdot\omega_0,h^{-1}\cdot g_0)$ for all $t$.
\end{lemma}

\begin{remark}
For simplicity, given a time-dependent function $A=A(t)$, we will sometimes write $A'$ instead of $\ddt A$.
\end{remark}

\begin{proof}
We first note that the equality $P+Q^{ac} =P^{ac}+Q$ follows from \eqref{comp}.  By using \eqref{OmGh} and that the transpose of $P^{ac}$ and $Q$ with respect to $g_0$ respectively satisfy that $(P^{ac})^tG=-GP^{ac}$, $Q^tG=GQ$, we obtain
\begin{align*}
G'=&(h^t)'h+h^th' = -(P^{ac}+Q)^th^th -h^th(P^{ac}+Q) \\
=& -(P^{ac})^tG - Q^tG - GP^{ac} - G Q = -2G Q.
\end{align*}
On the other hand, we use in addition \eqref{Jh}, \eqref{comp}, \eqref{jota} and $P^{ac}J=-JP^{ac}$ to get
\begin{align*}
\Omega'=& -J_0(h^t)'J_0h -J_0h^tJ_0h' = J_0(P^{ac}+Q)^th^tJ_0h + J_0h^tJ_0h(P^{ac}+Q) \\
=& J_0(P^{ac})^tGJ + J_0Q^tGJ + J_0GJP^{ac} + J_0GJQ \\
=& J_0 \left( -GP^{ac}J + GQJ + GJP^{ac} + GJQ \right) = J_0\left( 2GJP^{ac} + 2GJQ^{c} \right) \\
=& -2\Omega(P^{ac}+Q^c) = -2\Omega(P^{ac}+P^c) = -2\Omega P,
\end{align*}
and thus $(h^{-1}\cdot\omega_0,h^{-1}\cdot g_0)$ is a $(p,q)$-flow solution by \eqref{eqLG2}.  Since it also starts at $(\omega_0,g_0)$, it must coincide with $(\omega,g)$ by uniqueness of the solution, concluding the proof of the lemma.
\end{proof}

An intriguing consequence of the above lemma is that only the value of $P+Q^{ac}$ determines the whole evolution.  This will play a key role throughout the rest of the paper, specially in the study of self-similar solutions.

\section{Bracket flow}\label{sec-BF}

We present in this section a general approach to study any $(p,q)$-flow on a Lie group.  Our main tool is a dynamical system defined on the variety of $2n$-dimensional Lie algebras, a set which parameterizes the space of all left-invariant almost-hermitian structures on all simply connected Lie groups of
dimension $2n$ (see \cite{minimal}).  This evolution is called the {\it bracket flow} and is proved to be equivalent in a precise sense to the $(p,q)$-flow.  The approach is useful to better visualize the possible pointed limits of $(p,q)$-flow solutions, under diverse rescalings, as well as to address regularity issues.  Immortal, ancient and self-similar solutions arise naturally from the qualitative analysis of the bracket flow.  This approach has been used to study the Ricci flow on nilmanifolds in \cite{nilricciflow} and on homogeneous manifolds in \cite{homRF}.  On the other hand, N. Enrietti, A. Fino and L. Vezzoni study in \cite{EnrFnVzz} the pluriclosed flow for nilmanifolds and as an application of the bracket flow approach, they prove long-time existence for any solution.

For Lie algebras, two almost-hermitian structures $(\ggo_1,\omega_1,g_1)$,  $(\ggo_2,\omega_2,g_2)$ are said to be {\it equivalent} if there exists a Lie algebra isomorphism $\vp:\ggo_1\longrightarrow\ggo_2$ such that $\omega_2=\vp\cdot\omega_1$ and $g_2=\vp\cdot g_1$.  This implies the usual equivalence between the corresponding left-invariant almost-hermitian structures on the simply connected Lie groups, and thus by the diffeomorphism invariance of $p$ and $q$, we have that
\begin{equation}\label{inv-pq}
P(\omega_2,g_2)=\vp P(\omega_1,g_1)\vp^{-1}, \qquad Q(\omega_2,g_2)=\vp Q(\omega_1,g_1)\vp^{-1}.
\end{equation}
The same holds for $P^c$, $P^{ac}$, $Q^c$ and $Q^{ac}$ by using that $J_2=\vp J_1\vp^{-1}$.

If $(\omega(t),g(t))$ is a $(p,q)$-flow solution starting at $(\omega_0,g_0)$, then by Lemma \ref{h} the family $h=h(t)\in\Gl(\ggo)$ defined there satisfies that
$$
h:(\lb,\omega,g)\longrightarrow (\mu,\omega_0,g_0), \qquad\mbox{where}\qquad \mu=\mu(t):=h\cdot\lb=h[h^{-1}\cdot,h^{-1}\cdot],
$$
is an equivalence of almost-hermitian structures for all $t$.  Here $\lb$ denotes the Lie bracket of the Lie algebra $\ggo$ and so $\mu$ defines a new Lie algebra with same underlying vector space $\ggo$, which is isomorphic to $(\ggo,\lb)$ for all $t$.  The values of the corresponding curvature tensors of this new family of almost-hermitian structures $(\mu(t),\omega_0,g_0)$ will be denoted by $P_\mu$ and $Q_\mu$, which by \eqref{inv-pq} satisfy
\begin{equation}\label{PQ}
P_\mu=hPh^{-1}, \qquad Q_\mu^{ac}=hQ^{ac}h^{-1}, \qquad\forall t,
\end{equation}
where as always $P=P(\omega,g)$ and $Q=Q(\omega,g)$ for all $t$.

It is straightforward to check that $\ddt\mu=-\delta_\mu(h'h^{-1})$ if $\mu=h\cdot\lb$, where
$\delta_\mu:\End(\ggo)\longrightarrow\Lambda^2\ggo^*\otimes\ggo$ is given by
\begin{equation}\label{defdelta}
\delta_\mu(A):=\mu(A\cdot,\cdot)+\mu(\cdot,A\cdot)-A\mu(\cdot,\cdot), \qquad\forall A\in\End(\ggo).
\end{equation}
Notice that $\delta_\mu(A)=-\pi(A)\mu = -\ddt|_{t=0} e^{tA}\cdot\mu$ (see \eqref{defpi}).  It follows from Lemma \ref{h} and \eqref{PQ} that the one-parameter family of Lie brackets $\mu(t)$ evolves according to the following ODE:
\begin{equation}\label{BF}
\ddt\mu=\delta_\mu(P_\mu+Q_\mu^{ac}), \qquad\mu(0)=\lb.
\end{equation}
This equation will be called the {\it $(p,q)$-bracket flow}, and a natural question is if the understanding of its qualitative behavior and dynamical properties may provide new insights into the study of some curvature flows for homogeneous almost-hermitian structures and their self-similar solutions.  Since $\ddt\mu$ is tangent to the $\Gl(\ggo)$-orbit of $\mu$, by a standard ODE theory argument we obtain that $\mu(t)\in \Gl(\ggo)\cdot\lb$ for all $t$.  Our next result shows that any $(p,q)$-flow is equivalent in a precise way to its corresponding $(p,q)$-bracket flow.

For a given simply connected almost-hermitian Lie group $(G,\omega_0,g_0)$ with Lie algebra $\ggo$, let us consider the following two one-parameter families of almost-hermitian Lie groups:
\begin{equation}\label{3rm}
(G,\omega(t),g(t)), \qquad (G_{\mu(t)},\omega_0,g_0),
\end{equation}
where $(\omega(t),g(t))$ is the solution to the $(p,q)$-flow (\ref{eqLG}) starting at $(\omega_0,g_0)$ and $\mu(t)$ is the $(p,q)$-bracket flow (\ref{BF}) starting at the Lie bracket $\lb$ of $\ggo$.  Here $G_\mu$ denotes the simply connected Lie group with Lie algebra $(\ggo,\mu)$.

\begin{theorem}\label{eqfl}
There exist time-dependent Lie group isomorphisms $h(t):G\longrightarrow G_{\mu(t)}$ such that
$$
(\omega(t),g(t))=h(t)^*(\omega_0,g_0), \qquad\forall t,
$$
which can be chosen such that their derivatives at the identity, also denoted by $h=h(t)$, is the solution to any of the following systems of ODE's:
\begin{itemize}
\item[(i)] $\ddt h=-h(P+Q^{ac})=-h(P^{ac}+Q)$, $\quad h(0)=I$.
\item[ ]
\item[(ii)] $\ddt h=-(P_\mu+Q_\mu^{ac})h=-(P_\mu^{ac}+Q_\mu)h$, $\quad h(0)=I$.
\end{itemize}
The following conditions also hold for all $t$:
\begin{itemize}
\item[(iii)] $(\omega(t),g(t))=(h^{-1}\cdot\omega_0, h^{-1}\cdot g_0)$.
\item[ ]
\item[(iv)] $\mu(t)=h\cdot\lb$.
\end{itemize}
\end{theorem}

\begin{proof}
It has been shown above that part (i) implies all the other statements in the theorem by using Lemma \ref{h}.  Recall that any Lie algebra isomorphism can be lifted to a Lie group isomorphism between the respective simply connected Lie groups.  Let us then assume that part (ii) holds, and so $h(t)$ is defined on the same time interval as $\mu(t)$.  It follows that $\mu(t)=h(t)\cdot\lb$ for all $t$ as they both solve the same ODE \eqref{BF} and start at $\lb$, and so part (iv) holds.  Thus $h$ determines an equivalence between $(\lb,\wt{\omega}:=h^{-1}\cdot\omega_0,\wt{g}:=h^{-1}\cdot g_0)$ and $(\mu,\omega_0,g_0)$.  This implies that the corresponding curvature tensors satisfy $P_\mu=h\wt{P}h^{-1}$ and $Q_\mu^{ac}=h\wt{Q}^{ac}h^{-1}$, and consequently $h'=-h(\wt{P}+\wt{Q}^{ac})$, from which follows that $(\wt{\omega}(t),\wt{g}(t))$ is also a $(p,q)$-flow solution starting at $(\omega_0,g_0)$ by using Lemma \ref{h}.  By uniqueness of the solution, parts (i) and (iii) follow, concluding the proof of the theorem.
\end{proof}

It is worth pointing out the following useful facts which are direct consequences of the theorem:

\begin{itemize}
\item The $(p,q)$-flow $(\omega(t),g(t))$ and the $(p,q)$-bracket flow $\mu(t)$ differ only by pullback by time-dependent diffeomorphisms.

\item They are equivalent in the following sense: each one can be obtained from the other by solving the corresponding ODE in part (i) or (ii) and applying either part (iv) or (iii), accordingly.

\item The maximal interval of time where a solution exists is therefore the same for both flows, and so regularity issues can be directly addressed on the $(p,q)$-bracket flow.

\item At each time $t$, the almost-hermitian manifolds in (\ref{3rm}) are equivalent, so that the behavior
of any class of curvature and of any other invariant along $(\omega(t),g(t))$ can be studied along the $(p,q)$-bracket flow.
\end{itemize}

\subsection{Pointed convergence}\label{conver}
If some sequence $\mu(t_k)$ (or a suitable normalization $c_k\mu(t_k)$) converges to a Lie bracket $\lambda$ as $t_k\to\infty$, then we can apply the results given in \cite[Section 6.4]{spacehm} to get convergence of the almost-hermitian manifolds $(G,\omega(t_k),g(t_k))\to (G_\lambda,\omega_0,g_0)$ relative to the pointed (or Cheeger-Gromov) topology.

More precisely, assume that a normalization $c_k\mu(t_k)\to\lambda$, as $k\to\infty$, with $t_k\to T_\pm$, where $(T_-,T_+)$ denotes the maximal interval of time existence.  By Theorem \ref{eqfl}, if $\vp_k:G\longrightarrow G_{c_k\mu(t_k)}$ is the isomorphism with derivative $\tfrac{1}{c_k}h(t_k)$, then
$$
\vp_k^*(\omega_0,g_0)=\left(\tfrac{1}{c_k^2}\omega(t_k),\tfrac{1}{c_k^2}g(t_k)\right).
$$
It now follows from \cite[Corollary 6.20]{spacehm} that, after possibly passing to a subsequence, the almost-hermitian manifolds $\left(G,\tfrac{1}{c_k^2}\omega(t_k),\tfrac{1}{c_k^2}g(t_k)\right)$ converges in the pointed sense to $(G_\lambda,\omega_0,g_0)$, as $k\to\infty$.  This means that there exist
\begin{itemize}
\item[ ] a sequence of open subsets $\Omega_k\subset G_\lambda$ which eventually contains every compact subset of $G_\lambda$;

\item[ ] a sequence of embeddings $\phi_k:\Omega_k\longrightarrow G$;
\end{itemize}
such that $\tfrac{1}{c_k^2}\phi_k^*\omega(t_k)\to \omega_0$ and $\tfrac{1}{c_k^2}\phi_k^*g(t_k)\to g_0$ smoothly, as $k\to\infty$, on every compact subset of $G_\lambda$.  Note that the points play no role in the pointed convergence by homogeneity.  By using charts with relatively compact domains which cover $G_\lambda$, smooth convergence can be defined as the partial derivatives $\partial^\alpha(\omega_k)_{ij}$, $\partial^\alpha(g_k)_{ij}$ of the coordinates $(\omega_k)_{ij}$, $(g_k)_{ij}$, respectively, converging to $\partial^\alpha (\omega_0)_{ij}$ and $\partial^\alpha (g_0)_{ij}$ uniformly, as $k\to\infty$, for every chart and every multiindex $\alpha$.

It is not necessary to pass to a subsequence in the nilpotent, or more general, in the {\it completely solvable} case (i.e. when all the eigenvalues of $\ad{X}$ are real for any $X\in\ggo$).

We note that one can obtain at most one limit up to scaling by considering different normalizations of the bracket flow.  More precisely, assume that $c(t)\mu(t)\to\lambda\ne 0$, as $t\to T_\pm$.  Then the limit $\tilde{\lambda}$ of any other converging normalization $a(t)\mu(t)$ necessarily satisfies $\tilde{\lambda}=c\lambda$ for some $c\in\RR$ (see \cite[Proposition 4.1,(iii)]{homRS}).  Recall that the above observation only concerns solutions which are not chaotic, in the sense that the $\omega$-limit is a single point.

\section{Regularity}\label{sec-reg}

In the presence of any geometric flow, a natural question is what is the simplest quantity that, as long as it remains bounded, it prevents the formation of a singularity.  Many regularity results of this kind have been obtained for the hermitian curvature flow (see \cite{StrTn1}), the pluriclosed flow (see \cite{StrTn3}), the Chern-Ricci flow (see \cite{TstWnk}) and the symplectic curvature flow (see \cite{StrTn2}).  In this section, as an application of the $(p,q)$-bracket flow approach developed in Section \ref{sec-BF}, we obtain some regularity results for a general $(p,q)$-flow on a Lie group.

In addition to diffeomorphism invariance, in what follows, we shall assume that $p$ and $q$ are also scaling invariants:
$$
p(c\omega,cg)=p(\omega,g), \qquad q(c\omega,cg)=q(\omega,g), \qquad\forall c\in\RR^*.
$$
This condition actually holds for most of the curvature tensors considered in different evolution equations in the literature, including all the ones mentioned in Examples \ref{ex1}, \ref{ex2} and \ref{ex3}, and it is equivalent to
\begin{equation}\label{PQc}
P(c\omega,cg)=\tfrac{1}{c}P(\omega,g),  \qquad Q(c\omega,cg)=\tfrac{1}{c}Q(\omega,g), \qquad\forall c\in\RR^*.
\end{equation}

By using that for any $c\ne 0$,
$$
\tfrac{1}{c}I:(\mu,\tfrac{1}{c^2}\omega_0,\tfrac{1}{c^2}g_0)\longrightarrow (c\mu,\omega_0,g_0),
$$
is an equivalence of almost-hermitian structures and \eqref{PQc}, we obtain that the operators defined in \eqref{PQ} satisfy
\begin{equation}\label{PQmuc}
P_{c\mu}+Q_{c\mu}^{ac}=c^2(P_\mu+Q_\mu^{ac}), \qquad\forall c\in\RR.
\end{equation}

Let $(T_-,T_+)$ denote the maximal interval of time existence for the $(p,q)$-bracket flow solution $\mu(t)$ (see \eqref{BF}), or equivalently, of the $(p,q)$-flow solution $(\omega(t),g(t))$ starting at a left-invariant almost-hermitian structure $(\ggo,\omega_0,g_0)$, with $-\infty\leq T_-<0<T_+\leq\infty$.

The norm $|\mu|$ of a Lie bracket will be defined in terms of the canonical inner product on $\Lambda^2\ggo^*\otimes\ggo$ given by
\begin{equation}\label{ipg}
\la\mu,\lambda\ra:=\sum g_0(\mu(e_i,e_j),\lambda(e_i,e_j)) =\sum\mu_{ij}^k\lambda_{ij}^k,
\end{equation}
where $\{ e_i\}$ is any orthonormal basis of $(\ggo,g_0)$ and the structural constants are defined by $\mu(e_i,e_j)=\sum\mu_{ij}^ke_k$.  A natural inner product on $\End(\ggo)$ is also determined by $g_0$ by $\la A,B\ra:=\tr{AB^t}$.

The proof of the following proposition is strongly based on the arguments used by R. Lafuente in \cite{Lfn} to prove that the scalar curvature controls the formation of singularities of homogeneous Ricci flows.

\begin{proposition}\label{prop-reg}
If $T_+$ is finite (resp. $T_-$), then
\begin{itemize}
\item[(i)] $|\mu(t)|\geq\frac{C}{(T_+-t)^{1/2}}$, $\forall t\in[0,T_+)$ (resp. $|\mu(t)|\geq\frac{C}{(t-T_-)^{1/2}}$, $\forall t\in(T_-,0]$), for some positive constant $C$ depending only on $n$.
\item[ ]
\item[(ii)] $\int_0^{T_+} |P_\mu+Q_\mu^{ac}|\; dt =\infty$ (resp. $\int_{T_-}^0 |P_\mu+Q_\mu^{ac}|\; dt =\infty$).
\end{itemize}
\end{proposition}

\begin{proof}
Assume that $T_+<\infty$ (the proof for $-\infty<T_-$ is completely analogous).  It follows from \eqref{PQmuc} that
\begin{equation}\label{reg1}
\left|\ddt\mu\right|\leq C_1|P_\mu+Q_\mu^{ac}| |\mu| \leq C_2|\mu|^3, \qquad\forall t,
\end{equation}
for some constants $C_1,C_2>0$ depending only on $n$.  This implies that $\ddt|\mu|^2\leq 2C_2|\mu|^4$, and so for any $t_0\in[0,T_+)$,
$$
|\mu(t)|^2\leq\frac{1}{-2C_2(t-t_0)+|\mu(t_0)|^{-2}}, \qquad\forall t\in[t_0,T_+).
$$
Thus $T_+\geq t_0+\tfrac{|\mu(t_0)|^{-2}}{2C_2}$ since $|\mu(t)|$ must blow up at a singularity, from which part (i) follows.

By using \eqref{reg1} one also obtains that
$$
\ddt|\mu|^2\leq 2C_1|P_\mu+Q_\mu^{ac}||\mu|^2, \qquad \forall t,
$$
and therefore part (ii) follows from
$$
2C_1\int_0^s |P_\mu+Q_\mu^{ac}| \; dt\geq \log{|\mu(s)|^2}-\log{|\mu_0|^2}, \qquad\forall s\in[0,T_+),
$$
concluding the proof of the proposition.
\end{proof}

The following corollary follows from Theorem \ref{eqfl}.

\begin{corollary}\label{cor-reg}
If a left-invariant $(p,q)$-flow solution $(\omega(t),g(t))$ on a Lie group has a finite singularity at $T_+$ (resp. $T_-$), then
$$
\int_0^{T_+} |P+Q^{ac}|\; dt =\infty \qquad \left(\mbox{resp.}\quad \int_{T_-}^0 |P+Q^{ac}|\; dt =\infty\right).
$$
\end{corollary}

This was proved for the pluriclosed flow of any compact manifold in \cite[Theorem 1.2]{StrTn3}, and may be considered as the analogous to N. Sesum's result on the Ricci flow (see \cite{Ssm}).

\begin{remark}\label{PQacort}
It follows from the triangular inequality that the integral of at least one of $|P|$, $|Q^{ac}|$ must blow up in a finite-time singularity.  Actually, as $P$ and $Q^{ac}$ are orthogonal (indeed, $\la P,Q^{ac}\ra=\la \unm(P+P^t),Q^{ac}\ra=\la P^c,Q^{ac}\ra=0$), one has that
$$
|P+Q^{ac}|=\left(|P|^2+|Q^{ac}|^2\right)^{1/2}.
$$
\end{remark}

We now compute the evolution of the Ricci and scalar curvatures of $g(t)$ along a $(p,q)$-solution $(\omega(t),g(t))$.

According to \cite[7.38]{Bss} (see also \cite[Section 2.3]{homRF}), the Ricci operator $\Ricci_\mu$ of $(G_\mu,g)$ is given by,
\begin{equation}\label{ricci}
\Ricci_\mu=\mm_{\mu}-\unm B_{\mu}-S(\ad_{\mu}{H_{\mu}}),
\end{equation}
where
\begin{equation}\label{Rmm}
\la\mm_{\mu},E\ra=-\unc\la\delta_\mu(E),\mu\ra, \qquad\forall E\in\End(\ggo),
\end{equation}
$B_{\mu}$ is the Killing form
$$
g(B_{\mu}X,Y)=\tr{\ad_\mu{X}\ad_\mu{Y}},\qquad\forall X,Y\in\ggo,
$$
$H_{\mu}\in\ggo$ is defined by
$$
g(H_{\mu},X)=\tr{\ad_\mu{X}}, \qquad\forall X\in\ggo,
$$
and
\begin{equation}\label{sym}
S:\End(\ggo)\longrightarrow\End(\ggo), \qquad S(A):=\unm(A+A^t),
\end{equation}
is the symmetric part of an operator.  The scalar curvature is therefore given by
$$
R_\mu=-\unc|\mu|^2-\unm\tr{B_\mu}-|H_\mu|^2.
$$

\begin{proposition}\label{eqs}
Let $\mu(t)$ be a $(p,q)$-bracket flow solution.
\begin{itemize}
\item[(i)] The Ricci operator $\Ricci=\Ricci_{\mu(t)}$ evolves by
\begin{align*}
\ddt\Ricci =& -\unm\Delta(P+Q^{ac})-\unm(B(P+Q^{ac})+(P+Q^{ac})^tB) \\
& -2S(\ad_{\mu}{S(P+Q^{ac})(H)}) - S([\ad_{\mu}{H},P+Q^{ac}]),
\end{align*}
where $\Delta:=S\circ\delta_\mu^t\delta_\mu$.
\item[ ]
\item[(ii)] The evolution of the scalar curvature $R=R_{\mu(t)}$ is given by
$$
\ddt R=2\la P+Q^{ac},\Ricci\ra +2\la P+Q^{ac},S(ad_\mu{H})\ra -2\la (P+Q^{ac})H,H\ra.
$$
\item[ ]
\item[(iii)] $\ddt |\mu|^2=-8\la P+Q^{ac},\mm\ra$.
\end{itemize}
\end{proposition}

\begin{proof}
We use formulas (36), (37) and (38) from \cite{homRF} to prove part (i) as follows:
\begin{align*}
\ddt\Ricci =& \ddt\mm -\unm\ddt B -\ddt S(\ad_\mu{H}) \\
=&  d \mm|_{\mu}\delta_{\mu}(P+Q^{ac}) -\unm dB|_\mu\delta_\mu(P+Q^{ac}) \\
& -S\left(\ad_{\delta_\mu(P+Q^{ac})}{H_\mu}
+\ad_\mu{dH|_\mu\delta_\mu(P+Q^{ac})}\right) \\
=&-\unm\Delta(P+Q^{ac}) -\unm(B(P+Q^{ac})+(P+Q^{ac})^tB) \\
& -S\left(\ad_\mu{(P+Q^{ac})H_\mu}+[\ad_\mu{H_\mu},P+Q^{ac}]+\ad_\mu{(P+Q^{ac})^tH_\mu}\right).
\end{align*}

We now use part (i) to prove (ii):
\begin{align*}
\ddt R =& \tr{\ddt \Ricci} = -\unm\tr{\delta_{\mu}^t\delta_{\mu}(P+Q^{ac})} -\tr{B(P+Q^{ac})} -2\tr{\ad_{\mu}{S(P+Q^{ac})(H)}} \\
=& -\unm\la \delta_\mu(P+Q^{ac}),\delta_\mu(I)\ra -\la P+Q^{ac},B\ra -2\la (P+Q^{ac})(H),H\ra \\
=&  2\la P+Q^{ac},\mm\ra -2\la P+Q^{ac},\unm B\ra -2\la P+Q^{ac},S(\ad_{\mu}{H})\ra \\
& +2\la P+Q^{ac},S(\ad_{\mu}{H})\ra -2\la (P+Q^{ac})(H),H\ra.
\end{align*}

Finally, we prove part (iii):
$$
\ddt |\mu|^2= 2\la\ddt\mu,\mu\ra =2\la\delta_{\mu}(P+Q^{ac}),\mu\ra = -8\la (P+Q^{ac},\mm\ra,
$$
concluding the proof of the proposition.
\end{proof}

Recall that only unimodular Lie groups can admit lattices, and since $G_\mu$ is unimodular if and only if $H_\mu=0$, one obtains in this case much simpler evolutions for $\Ricci$ and $R$ in parts (i) and (ii) of the above theorem, respectively.

\section{Self-similar solutions}\label{sec-self}

As the equivalence class of an almost-hermitian structure $(\omega_0,g_0)$ on a Lie algebra $\ggo$ is $\Aut(\ggo)\cdot(\omega_0,g_0)$, it is natural to consider a $(p,q)$-flow solution to be self-similar if it stays in this orbit, up to scaling, for all time $t$.  It follows from \eqref{inv-pq} that
$$
P(h\cdot\omega,h\cdot g)=h P(\omega,g)h^{-1}, \qquad Q(h\cdot\omega,h\cdot g)=h Q(\omega,g)h^{-1},
$$
for any $h\in\Aut(\ggo)$ (see \eqref{inv-pq}), and the same holds for the $J$-invariant and anti-J-invariant components of $P$ and $Q$.  Condition \eqref{PQc} is also assumed to hold in this section.

\begin{lemma}\label{self}
Let $(\omega(t),g(t))$ be the $(p,q)$-flow solution starting at $(\omega_0,g_0)$, and consider $c(t)>0$, $\vp(t)\in\Aut(\ggo)$, both differentiable on $t$, with $c(0)=1$, $\vp(0)=I$ and $D:=\vp'(0)\in\Der(\ggo)$.  Then the following conditions are equivalent.

\begin{itemize}
\item[(i)] The solution has the form:
$$
\left\{\begin{array}{l}
\omega(t)=c(t)\omega_0(\vp(t)\cdot,\vp(t)\cdot), \\ \\
g(t)=c(t)g_0(\vp(t)\cdot,\vp(t)\cdot).
\end{array}\right.
$$

\item[(ii)] The starting curvature tensors satisfy
$$
\left\{\begin{array}{l}
c'(0)I+\left(D-J_0D^tJ_0\right)= -2P(\omega_0,g_0), \\ \\
c'(0)I+\left(D+D^t\right)= -2 Q(\omega_0,g_0).
\end{array}\right.
$$
\item[(iii)] The solution is precisely given by
$$
\left\{\begin{array}{l}
\omega(t)=(c'(0)t+1)\omega_0( e^{s(t)D}\cdot,e^{s(t)D}\cdot),  \\ \\
g(t)=(c'(0)t+1) g_0( e^{s(t)D}\cdot,e^{s(t)D}\cdot),
\end{array}\right.
$$
where $s(t):=\frac{\log(c'(0)t+1)}{c'(0)}$ if $c'(0)\ne 0$ and $s(t)=t$ when $c'(0)=0$.
\end{itemize}
\end{lemma}

\begin{proof}
It can be easily checked that part (i) is equivalent to
\begin{equation}\label{self2}
\left\{\begin{array}{l}
c'I+c\left(\vp^{-1}\vp'+(\vp^{-1}\vp')^{t_{\omega_0}}\right)= -2P(\omega_0,g_0), \qquad\forall t, \\ \\
c'I+c\left(\vp^{-1}\vp'+(\vp^{-1}\vp')^t\right)= -2Q(\omega_0,g_0), \qquad\forall t,
\end{array} \right.
\end{equation}
from which part (ii) follows by just evaluating at $t=0$.

We now assume (ii) and consider $(\omega,g)$ defined as in part (iii), that is, with $c(t)=c'(0)t+1$ and $\vp(t)=e^{s(t)D}$.  By using that $\vp^{-1}\vp'= s'\vp^{-1}D\vp=\tfrac{1}{c}D$, one obtains that condition \eqref{self2} holds for $(\omega,g)$ and so part (iii) follows.  This concludes the proof of the lemma as part (i) follows trivially from (iii).
\end{proof}

The following definition is motivated by Lemma \ref{self} and the terminology used in Ricci flow theory.

\begin{definition}\label{self-def}
An almost-hermitian structure $(\omega,g)$ on a Lie algebra $\ggo$ is called a $(p,q)$-{\it soliton} if for some $c\in\RR$ and $D\in\Der(\ggo)$,
$$
\left\{
\begin{array}{l}
P(\omega,g)=cI+\unm(D-JD^tJ), \\ \\
 Q(\omega,g)=cI+\unm(D+D^t).
\end{array}\right.
$$
We call $c$ the {\it cosmological constant} of the soliton.
\end{definition}

If $X_D$ is the vector field on the Lie group defined by the one-parameter subgroup of automorphisms $\vp_t$ with derivative $e^{tD}\in\Aut(\ggo)$, that is, $X_D(x)=\ddt|_0\vp_t(x)$ for
any $x\in G$, then $(\omega,g)$ is a $(p,q)$-soliton if and only if,
$$
\left\{
\begin{array}{l}
p(\omega,g)=c\omega+\unm(\omega(D\cdot,\cdot)+\omega(\cdot,D\cdot)) =c\omega-\unm\lca_{X_D}\omega, \\ \\
q(\omega,g)=cg+\unm(g(D\cdot,\cdot)+g(\cdot,D\cdot)) =cg-\unm\lca_{X_D}g,
\end{array}\right.
$$
where $\lca_{X}$ denotes Lie derivative.  The $(p,q)$-flow solution starting at a $(p,q)$-soliton $(\omega,g)$ is therefore given by
\begin{equation}\label{evsol}
\left\{
\begin{array}{l}
\omega(t)=(-2ct+1) e^{s(t)D}\cdot\omega, \\ \\
g(t)=(-2ct+1) e^{s(t)D}\cdot g,
\end{array}\right.
\end{equation}
where $s(t):=\frac{\log(-2ct+1)}{-2c}$ if $c\ne 0$ and $s(t)=t$ when $c=0$ (see Lemma \ref{self}).  We note that $(\omega,g)$ is defined as long as $-2ct+1>0$, thus obtaining
\begin{equation}\label{solint}
(T_-,T_+)=\left\{\begin{array}{lcl} (\tfrac{1}{2c},\infty), && c<0, \\ (-\infty,\tfrac{1}{2c}), && c>0, \\ (-\infty,\infty), && c=0. \end{array}\right.
\end{equation}

\begin{remark}\label{soliton}
In the general case, an almost-hermitian manifold $(M,\omega,g)$ will flow self-similarly according to the $(p,q)$-flow \eqref{eq}, in the sense that
$$
(\omega(t),g(t))=(c(t)\vp(t)^*\omega_0,c(t)\vp(t)^*g_0),
$$
for some $c(t)>0$ and $\vp(t)\in\Diff(M)$, if and only if
$$
\left\{
\begin{array}{l}
p(\omega,g)=c\omega+\lca_{X}\omega, \\ \\
q(\omega,g)=cg+\lca_{X}g,
\end{array}\right.
$$
where $c=-\unm c'(0)$ and the family $\psi(s)$ generated by the vector field $X$ on $M$ satisfies that $\vp(t)=\psi(-\unm s(t))$.  In analogy to the Ricci flow, this would be an appropriate definition of {\it soliton} for a given $(p,q)$-flow on almost-hermitian manifolds.  We note that this notion, when applied to left-invariant structures on Lie groups, is weaker than Definition \ref{self-def} (see Examples \ref{semi-CRF} and \ref{semi-SCF} for CRF and SCF, respectively).  However, if the full symmetry group $G$ of a soliton almost-hermitian structure $(M,\omega,g)$ (i.e. $G$ the set of all isometries which also preserve $\omega$) is transitive and $K$ is the isotropy subgroup at some point of $M$, then $M=G/K$ and one can prove in much the same way as in \cite[Theorem 3.1]{Jbl3} that there exists a one-parameter family of equivariant diffeomorphisms $\phi_t\in\Aut(G/K)$ (i.e. automorphisms of $G$ taking $K$ onto $K$) such that $(\omega(t),g(t))=(c(t)\phi_t^*g,c(t)\phi_t^*g)$ is a solution to the $(p,q)$-flow starting at $(\omega,g)$.  In the Ricci flow case, these special homogeneous solitons are called {\it semi-algebraic} (see \cite{Jbl3,homRS} for further information).
\end{remark}

If an almost-hermitian structure $(\omega,g)$ satisfies that
\begin{equation}\label{PQac}
\left\{\begin{array}{l}
P=c_1I+D_1, \\ \\
Q^{ac}=c_2I+D_2,
\end{array}\right.
\end{equation}
for some $c_i\in\RR$, $D_i\in\Der(\ggo)$, then $(\omega,g)$ is a $(p,q)$-soliton with $c=c_1+c_2$ and $D=D_1+D_2$.  This easily follows by using that $-JD_1^tJ=D_1^{t_{\omega}}=D_1$, $c_2I+D_2^c=0$, $D_2^t=D_2$ and \eqref{comp}.  However, the following apparently weaker condition, suggested by the strong relationship between a $(p,q)$-flow and its $(p,q)$-bracket flow given in Theorem \ref{eqfl}, is actually enough to get a soliton.

\begin{proposition}\label{PmasQac}
If an almost-hermitian structure $(\omega,g)$ satisfies that for some $c\in\RR$ and $D\in\Der(\ggo)$,
\begin{equation}\label{algsol}
P+Q^{ac}=cI+D,
\end{equation}
then $(\omega,g)$ is a $(p,q)$-soliton with same $c$ and $D$.
\end{proposition}

\begin{proof}
By taking transpose of both sides in equality $P+Q^{ac}=cI+D$ relative to $g$ and $\omega$, we respectively obtain that
$$
-JPJ+Q^{ac}=cI+D^t, \qquad P-Q^{ac}=cI-JD^tJ.
$$
The formulas for $P$ and $Q$ in Definition \ref{self-def} now follow by averaging these two equations with the above one and using \eqref{comp}, concluding the proof.
\end{proof}

\subsection{Bracket flow evolution of solitons}\label{BFsoliton}
We now study how $(p,q)$-solitons evolve according to the $(p,q)$-bracket flow.  We refer to \cite[Section 4]{homRS} for an analysis of the same question in the case of the Ricci flow, where a more detailed treatment is given to all of the claims made below.

It is easy to check that the $(p,q)$-bracket flow solution starting at a $(p,q)$-soliton for which condition \eqref{algsol} holds stays in a straight line, it is simply given by
$$
\mu(t)=(-2ct+1)^{-1/2}\lb.
$$
This property actually characterizes the class of solitons defined by \eqref{algsol}.  Moreover, it also easily follows that they are precisely the fixed points of any (possibly normalized) $(p,q)$-bracket flow, and consequently, the only possible limits, backward and forward, of any normalized $(p,q)$-bracket flow solution $a(t)\mu(t)$. This fact and the equivalence between the $(p,q)$-flow and its the $(p,q)$-bracket flow (see Theorem \ref{eqfl}) suggest that solitons satisfying \eqref{algsol} might exhaust the class of all $(p,q)$-solitons (up to equivalence).

The converse of Proposition \ref{PmasQac} might fail.  Indeed, if $(\omega,g)$ is a $(p,q)$-soliton for the pair $(c,D)$, then it is straightforward to prove that \begin{equation}\label{PQA}
P+Q^{ac}=cI+D-A, \qquad\mbox{where}\qquad A:=\unm\left(D-D^t\right)^c.
\end{equation}
Note that $A$ is a skew-hermitian map which is not necessarily a derivation.  In this general case, the $(p,q)$-bracket flow evolves by
$$
\mu(t)=(-2ct+1)^{-1/2} e^{s(t)A}\cdot\lb.
$$
Indeed, according to Theorem \ref{eqfl}, we must solve
$$
\ddt h=-h(P+Q^{ac}) = -(-2ct+1)^{-1} h\left(cI+D-e^{s(t)D}Ae^{-s(t)D}\right),
$$
(see \eqref{evsol} and \eqref{PQc}) which gives $h(t)=(-2ct+1)^{1/2}\left(e^{s(t)A}e^{-s(t)D}\right)$, and so the formula for $\mu(t)$ follows.

We note that the normalized solution $\lambda(t):=\mu(t)/|\mu(t)|$ and all the limits of subsequences $\lambda(t_k)$ are in the compact orbit $\U(n)\cdot\lb$, up to scaling (recall that $A\in\ug(n)$).  On the other hand, by Kronecker's theorem, there exists a sequence $t_k$, with $t_k \rightarrow \infty$, such that $e^{t_kA} \rightarrow I$.  This implies that $\lambda(t_0+t_k) \underset{k \rightarrow \infty}\longrightarrow \lambda(t_0)$ for any $t_0\in\RR$ and thus the whole normalized solution is contained in its $\omega$-limit.  The absence of this kind of chaotic behavior for the $(p,q)$-bracket flow would imply that $A$ is a derivation of $\ggo$, and consequently, condition \eqref{algsol} would necessarily hold for any $(p,q)$-soliton.  This is analogous to the question in the Ricci flow case whether every semi-algebraic soliton is algebraic (see \cite{homRS}), which has been resolved, in the affirmative, by M. Jablonski in \cite{Jbl2}.

\section{Chern-Ricci flow}\label{sec-CRF}

Let $(M,J)$ be a complex manifold.  Given a hermitian metric $g$, the Ricci flow starting at $g$ may not preserve the hermitian condition since the Ricci tensor is not in general a $(1,1)$-tensor beyond the K\"ahler case.  A natural $(1,1)$-tensor to consider instead is the Chern-Ricci tensor $p(\cdot,J\cdot)$, where $p$ is the Chern-Ricci form (see Appendix \ref{chern}), and one obtains the so called {\it Chern-Ricci flow} (or CRF for short):
\begin{equation}\label{CRF}
\dpar \omega=-2p, \qquad\mbox{or equivalently}, \qquad \dpar g=-2p(\cdot,J\cdot),
\end{equation}
where as always $\omega=g(J\cdot,\cdot)$ is the fundamental form.  We refer to \cite{TstWnk} and the references therein for further information on this evolution equation.

A solution to \eqref{CRF} preserves the compatibility condition by \eqref{comp} (recall that $p=p^c$ and $q=p(\cdot,J\cdot)$), and it follows from \eqref{eqJ2} that
\begin{equation}\label{CRFJ}
J(t)\equiv J_0.
\end{equation}
In other words, the complex manifold remains fixed along the flow, and thus the hermitian condition is also preserved.  If $\omega$ is K\"ahler, then $p$ is precisely the Ricci form and CRF becomes the K\"ahler-Ricci flow.

We note that the corresponding Chern-Ricci operators satisfy $P=Q$ (see \eqref{defPQ}) and in particular, $P=P^t=P^c$ is a symmetric and hermitian map with respect to $(g,J)$.

\vs

Let us now consider the CRF on Lie groups.  In this case, the Chern-Ricci form is given by $p=-\unm\tr{J\ad{[\cdot,\cdot]}}+\unm\tr{\ad{J[\cdot,\cdot]}}$ (see \eqref{CRform}).  Since $p$ depends only on $J$, we obtain from \eqref{CRFJ} that along a solution to \eqref{CRF}, $p(t)\equiv p_0$.  This implies that the CRF-solution starting at $(\omega_0,g_0)$ is simply given by
\begin{equation}\label{CRFsol}
\omega(t)=\omega_0-2tp_0, \qquad\mbox{or equivalently}, \qquad g(t)=g_0-2tp_0(\cdot,J\cdot).
\end{equation}
Therefore,
$$
\omega(t)=\omega_0((I-2tP_0)\cdot,\cdot),
$$
and so the solution exists as long as the hermitian map $I-2tP_0$ is positive.  It follows that the maximal interval of time existence $(T_-,T_+)$ of $\omega(t)$ is given by
\begin{equation}\label{CRFint}
T_+=\left\{\begin{array}{lcl} \infty, & \quad\mbox{if}\;P_0\leq 0, \\ \\ 1/(2p_+), & \quad\mbox{otherwise,}\end{array}\right. \qquad
T_-=\left\{\begin{array}{lcl} -\infty, & \quad\mbox{if}\;P_0\geq 0, \\ \\ 1/(2p_-), & \quad\mbox{otherwise,}\end{array}\right.
\end{equation}
where $p_+$ is the maximum positive eigenvalue of the Chern-Ricci operator $P_0$ of $\omega_0$ and $p_-$ is the minimum negative eigenvalue (recall that $p_0=\omega_0(P_0\cdot,\cdot)$).

It is easy to see that
$$
P(t)=(I-2tP_0)^{-1}P_0,
$$
from which follows that the family $h(t)\in\Gl(\ggo)$ defined in Lemma \ref{h} is given by $h(t)=(I-2tP_0)^{1/2}$.  By Theorem \ref{eqfl}, (iv), the solution to the Chern-Ricci bracket flow $\ddt\mu=\delta_\mu(P_\mu)$ is given by
$$
\mu(t)=(I-2tP_0)^{1/2}\cdot\lb,
$$
and hence relative to any orthonormal basis $\{ e_1,\dots,e_{2n}\}$ of eigenvectors of $P_0$, say with eigenvalues $\{ p_1,\dots,p_{2n}\}$, the structure coefficients of $\mu(t)$ are given by
\begin{equation}\label{muijk}
\mu_{ij}^k(t)=\left(\frac{1-2tp_k}{(1-2tp_i)(1-2tp_j)}\right)^{1/2} c_{ij}^k,
\end{equation}
where $c_{ij}^k$ are the structure coefficients of the Lie bracket $\lb$ of $\ggo$ (i.e. $[e_i,e_j]=\sum c_{ij}^ke_k$).

This implies that that the normalized solution $\mu(t)/|\mu(t)|$ always converges, as $t\to T_{\pm}$, toward a Lie bracket $\lambda$ such that $(G_\lambda,\omega_0,g_0)$ is a Chern-Ricci soliton (see Section \ref{BFsoliton}).

The evolution of the Chern scalar curvature $\tr{P}$ is also easy to understand.  Indeed,
$$
\tr{P(t)}=\sum_{i=1}^{2n} \frac{p_i}{1-2tp_i}, \qquad \ddt\tr{P(t)}=\sum_{i=1}^{2n} \frac{2p_i^2}{(1-2tp_i)^2}\geq 0,
$$
and hence $\tr{P(t)}$ is strictly increasing unless $P(t)\equiv 0$ (i.e. $\omega(t)\equiv \omega_0$).  Moreover, $\tr{P}$ must blow up in finite time singularities: if $T_+<\infty$ (resp. $T_->-\infty$), then
$$
\int_0^{T_+}\tr{P(t)}\; dt=\infty \qquad \left(\mbox{resp.} \quad \int_{T_-}^0\tr{P(t)}\; dt=-\infty\right).
$$
This improves in the case of CRF the regularity result given in Corollary \ref{cor-reg} for a general $(p,q)$-flow.  We also note that nevertheless, $$
\tr{P(t)}\leq \frac{C}{T_+-t},
$$
for some constant $C>0$, which is the claim of a well-known general conjecture for the K\"ahler-Ricci flow (see e.g. \cite[Conjecture 7.7]{SngWnk}.

From Definition \ref{self-def} and the fact that $P=Q$, we obtain that $(\omega,g)$ is a Chern-Ricci soliton if and only if its Chern-Ricci operator satisfies
\begin{equation}\label{CRsoldef}
P=cI+\unm(D+D^t), \qquad\mbox{for some}\quad c\in\RR, \quad D=D^c\in\Der(\ggo).
\end{equation}

It follows from \eqref{solint} and \eqref{CRFint} that either $P\leq 0$ ($c\leq 0$) or $P\geq 0$ ($c\geq 0$) for any Chern-Ricci soliton; in particular, $c=0$ if and only if $(\omega,g)$ is Chern-Ricci flat.  We next show that much more can be said on $P$ for Chern-Ricci solitons.

\begin{proposition}\label{CR-sol}
If a left-invariant hermitian structure $(\omega,g)$ on a Lie group is a Chern-Ricci soliton with cosmological constant $c\in\RR$, then its Chern-Ricci operator $P$ can have only $0$ and $c$ as eigenvalues.
\end{proposition}

\begin{proof}
According to \eqref{evsol}, the evolution is given by $\omega(t)=(-2ct+1)e^{s(t)D}\cdot\omega$, but since \eqref{CRFsol} must also hold, it is straightforward to obtain that
$$
e^{-2s(t)P}=I-2tP, \qquad\forall t.
$$
This implies that $e^{-2s(t)p_i}=1-2tp_i$ for any eigenvalue $p_i$ of $P$, from which we obtain that $cp_i-p_i^2=0$ by taking second derivatives.
\end{proof}

The following structural results for Chern-Ricci solitons, which are in particular valid for K\"ahler-Ricci solitons, may provide a starting point for approaching the classification problem.

\begin{proposition}\label{CR-sol2}
Let $(G,\omega,g)$ be a hermitian Lie group with Lie algebra $\ggo$ and Chern-Ricci operator $P\ne 0$.  Then the following conditions are equivalent.

\begin{itemize}
\item[(i)] $\omega$ is a Chern-Ricci soliton with cosmological constant $c$.
\item[ ]
\item[(ii)] $P=cI+D$, for some $D\in\Der(\ggo)$.
\item[ ]
\item[(iii)] The eigenvalues of $P$ are all either equal to $0$ or $c$, the kernel $\kg=\Ker{P}$ is an abelian ideal of $\ggo$ and its orthogonal complement $\kg^\perp$ (i.e. the $c$-eigenspace of $P$) is a Lie subalgebra of $\ggo$ (in particular, $\ggo$ is the semidirect product $\ggo=\kg^\perp\ltimes\kg$ and $c$ is nonzero).
\end{itemize}
\end{proposition}

\begin{proof}
As $\mu(t)/|\mu(t)|$ always converges as $t\to T_\pm$, the Chern-Ricci bracket flow never develops the chaotic behavior described at the end of Section \ref{BFsoliton}, and so any Chern-Ricci soliton must satisfy \eqref{algsol}.  This shows that part (ii) follows from (i) (the converse is trivial).

If we assume (ii), then the spectrum of $P$ is contained in $\{0,c\}$ by Proposition \ref{CR-sol}.  The rest of part (iii) easily follows by using that $P-cI\in\Der(\ggo)$.  Conversely, if (iii) holds, then $P-cI$ is clearly a derivation of $\ggo$ and so part (ii) follows.
\end{proof}

\begin{example}\label{semi-CRF}
The direct product $G=G_1\times G_2$ of a hermitian Lie group $G_1$ with Chern-Ricci operator $P_1=c_1I+D_1\ne 0$, $c_1\in\RR$, $D_1\in\Der(\ggo_1)$ and a Chern-Ricci flat hermitian nonabelian solvable Lie group $G_2$ (i.e. $P_2=0$) is a soliton in the more general sense as in Remark \ref{soliton} by defining the diffeomorphisms $\vp(t)$ to be $e^{s(t)D_1}$ on $G_1$ and the identity on $G_2$.  However, $G$ is not a Chern-Ricci soliton as in \eqref{CRsoldef}.  Indeed, if $P=cI+\unm(D+D^t)$ for some $c\in\RR$ and $D\in\Der(\ggo)$, then $D|_{\ggo_2}$ is normal and so its transpose is also a derivation, which implies that $c=0$ as $\ggo_2$ is nonabelian.  This implies that $P_1=\unm(D|_{\ggo_1}+(D|_{\ggo_1})^t)$ and hence $P_1=0$ by Proposition \ref{CR-sol2}, a contradiction.
\end{example}

It is proved in \cite{CRF} that $P=0$ if $G$ is nilpotent, and thus any invariant hermitian structure on a compact nilmanifold
is a fixed point for the CRF.  It is also studied in \cite{CRF} how is the limit $\lambda_\pm$ related to the starting point, and the existence problem for Chern-Ricci solitons on $4$-dimensional solvable Lie groups.

\section{Symplectic curvature flow}\label{sec-SCF}

Let $(M,\omega,g,J)$ be an almost-K\"ahler manifold, i.e. an almost-hermitian manifold such that $d\omega=0$.  With K\"ahler-Ricci flow as a motivation, it is natural to evolve the symplectic structure $\omega$ in the direction of the Chern-Ricci form $p$ (see Appendix \ref{chern}), but since in general $p\ne p^c$, one is forced to flow the metric $g$ as well in order to preserve compatibility.  One may therefore consider the following $(p,q)$-flow equation for a one-parameter family $(\omega(t),g(t))$ of almost-K\"ahler structures:
\begin{equation}\label{SCFeq}
\left\{\begin{array}{l}
\dpar\omega=-2p, \\ \\
\dpar g=-2p^c(\cdot,J\cdot)-2\ricac,
\end{array}\right.
\end{equation}
where $p$ is the Chern-Ricci form of $(\omega,g)$ and $\ricci$ is the Ricci tensor of $g$.  This equation was introduced in \cite{StrTn2} and is called the {\it symplectic curvature flow} (or SCF for short).  It is evident that SCF preserves the compatibility condition \eqref{comp} (note that $Q=P^c+\Ricci^{ac}$) and the almost-K\"ahler condition (recall that $dp=0$).  According to \eqref{eqJ2}, SCF makes $J$ to evolve as follows:
$$
\dpar J=-2JP^{ac}-2J\Ricac = -2JP^{ac} + [\Ricci,J],
$$
where $\Ricci$ denotes the Ricci operator of the metric $g$ (i.e. $\ricci=g(\Ricci\cdot,\cdot)$).  We note that if $J_0$ is integrable, i.e. $(\omega_0,g_0)$ K\"ahler, then $J=J_0$, $\ricac=0$ and $p^c(\cdot,J\cdot)=\ricci$ for all $t$, and so SCF becomes precisely the K\"ahler-Ricci flow for $g(t)$.

Let $(\omega(t),g(t))$ be a SCF-solution and assume that $M$ is compact.  In analogy with the theorem of Tian-Zhang for the K\"ahler-Ricci flow (see \cite{TnZhn}), it is conjectured in \cite{StrTn2} that $(\omega(t),g(t))$ exists smoothly as long as the cohomology class $[\omega(t)]\in H^2(M,\RR)$ belongs to the cone $\cca$ of all classes which can be represented by a symplectic form.

Since $\omega(t)$ is a deformation equivalence, all the corresponding time-dependent symplectic (or complex) vector bundles are pairwise isomorphic and thus the first Chern class is constant in time: $c_1(M,\omega(t))\equiv c_1(M,\omega_0)$.  Recall that $c_1(M,\omega)$ is well defined as $c_1(M,J)$ for any compatible almost-complex structure $J$ with same orientation.  By using that $[p]=2\pi c_1(M,\omega)$, one obtains that the class evolves by $\ddt [\omega]=-4\pi c_1(M,\omega_0)$, and hence
$$
[\omega(t)]=-4t\pi c_1(M,\omega_0)+[\omega_0].
$$
The conjecture mentioned above is therefore equivalent to say that the maximal existence time $T_+$ is given by
$$
T_+=\sup\{ t>0:[\omega(t)]=-4t\pi c_1(M,\omega_0)+[\omega_0]\in\cca\}.
$$
Assume now that $c_1(M,\omega_0)=0$.  The solution is therefore expected to be immortal, i.e. $T_+=\infty$.  Furthermore, since $[\omega(t)]\equiv [\omega_0]$, it follows from Moser Theorem (see e.g. \cite[Theorem 7.3]{Cnn} or \cite[Theorem 3.17]{McDSlm}) that
$$
\omega(t)=\vp(t)^*\omega_0, \qquad\mbox{for some}\qquad \vp(t)\in\Diff(M), \quad \vp(0)=id.
$$
In other words, under the vanishing of the first Chern class, the symplectic manifold $(M,\omega(t))$ does not really change along a SCF-solution $(\omega(t),g(t))$.  This in particular holds for invariant almost-K\"ahler structures on compact quotients $M=G/\Gamma$ of Lie groups, as the tangent bundle is in this case trivial and so $c_1(M)=0$.

Since $Q^{ac}=\Ricac$ for SCF, the bracket flow is given by
\begin{equation}\label{SCF-BF}
\ddt\mu=\delta_\mu(P_\mu+\Ricci_\mu^{ac}), \qquad\mu(0)=\lb,
\end{equation}
and the evolution of the scalar curvature and the square norm of the bracket, in the unimodular case, are respectively given by
$$
\ddt R=2\la P^c,\Ricci\ra + 2|\Ricac|^2, \qquad \ddt |\mu|^2=-8\la P^c+\Ricac,\mm\ra.
$$
Recall from \eqref{CRform} that the Chern-Ricci form of the almost-K\"ahler structure $(\mu,\omega_0,g_0,J_0)$ is given by
\begin{equation}\label{SCF-CRform}
p(X,Y)=-\unm\tr\left(J_0\ad_\mu{\mu(X,Y)}\right) + \unm\tr\left(\ad_\mu{J_0\mu(X,Y)}\right), \qquad\forall X,Y\in\ggo.
\end{equation}
We note that the second term vanishes for unimodular Lie groups, and that $p=0$ in the $2$-step nilpotent case.

According to Definition \ref{self-def}, we say that an almost-K\"ahler structure $(\omega,g)$ on a Lie algebra $\ggo$ is a {\it SCF-soliton} if for some $c\in\RR$ and $D\in\Der(\ggo)$,
\begin{equation}\label{SCF-sol}
\left\{
\begin{array}{l}
P=cI+\unm(D-JD^tJ), \\ \\
P^c+\Ricac=cI+\unm(D+D^t).
\end{array}\right.
\end{equation}
We note that a SCF-soliton $(\omega,g)$ is {\it static} (i.e. $p=c\omega$ and $\ricac=0$, and so they evolve by $(\omega(t),g(t))=(-2ct+1)(\omega,g)$) if and only if $D\in \ug(n)=\spg(\omega)\cap\sog(g)$, if and only if $e^{tD}\in U(n)=\Spe(\omega)\cap\Or(g)$ for all $t$.  The sufficient condition for a soliton given in \eqref{algsol} becomes
\begin{equation}\label{SCF-algsol}
P+\Ricac=cI+D.
\end{equation}
By using this condition, it has been found in \cite{SCFmuA} a SCF-soliton on most of symplectic $4$-dimensional Lie groups, where it is also studied the SCF-evolution of solvable Lie algebras with a codimension-one abelian ideal in any dimension.  E. Fern\'andez-Culma \cite{Frn} found a SCF-soliton on most $2$ and $3$-step symplectic nilpotent Lie groups of dimension $6$.

\begin{example}\label{semi-SCF}
The direct product $G=G_1\times G_2$ of an almost-K\"ahler nilpotent Lie group $G_1$ with $P_1=0$ (e.g. $2$-step nilpotent) and $\Ricci_1^{ac}=c_1I+D_1\ne 0$, $c_1\in\RR$, $D_1\in\Der(\ggo_1)$, and a flat almost-K\"ahler nonabelian Lie group $G_2$ (i.e. $P_2=0$ and $\Ricci_2^{ac}=0$) is a soliton in the more general sense as in Remark \ref{soliton} by defining the diffeomorphisms $\vp(t)$ to be $e^{s(t)D_1}$ on $G_1$ and the identity on $G_2$.  However, $G$ is not a SCF-soliton as in \eqref{SCF-sol}.  Indeed, if $P^c+\Ricac=cI+\unm(D+D^t)$, for some $c\in\RR$ and $D\in\Der(\ggo)$, then $D|_{\ggo_2}$ is normal and so its transpose is also a derivation, which implies that $c=0$ as $\ggo_2$ is nonabelian.  It follows from \eqref{SCF-sol} that $\Ricci_1^{ac}=((D|_{\ggo_1})^t)^{ac}$ and hence $\Ricci_1^{ac}=0$ as $\Ricci$ is orthogonal to any derivation by \cite[Lemma 6.1,(iii)]{nilricciflow}, a contradiction.
\end{example}

\subsection{Anti-complexified Ricci flow}
In the case when $p(\omega,g)=0$ for all time $t$, the SCF-solution $(\omega(t),g(t))$ satisfies that $\omega(t)\equiv\omega_0$ and $g(t)$ is a solution to the {\it anti-complexified Ricci flow} (acRF for short) studied in \cite{LeWng}, defined by
$$
\dpar g=-2\ricac.
$$
Thus SCF becomes in this case an evolution for compatible metrics on a fixed symplectic manifold $(M,\omega_0)$.  This for instance happens on any $2$-step nilmanifold.

Since $P=0$ and $Q^{ac}=\Ricci^{ac}$ along an acRF-solution, according to Proposition \ref{eqs}, (ii), the evolution of the scalar curvature in the unimodular case is simply given by
$$
\ddt R = 2|\Ricci^{ac}|^2.
$$
It follows from Corollary \ref{cor-reg} that $R$ must blow up in a finite-time singularity for acRF.  Indeed,
$$
\int_0^{T_+}|\Ricci^{ac}|\; dt \leq \int_0^{T_+}\frac{1+|\Ricci^{ac}|^2}{2}\; dt =\unm T_+ +\unc\lim_{t\to T_+} R -\unc R_0.
$$

\begin{proposition}\label{acRF-reg}
If a left-invariant acRF-solution $g(t)$ on a unimodular Lie group has a finite-time singularity at $T_+$ (resp. $T_-$), then
$$
\lim_{t\to T_+} R(g(t)) =\infty \qquad \left(\mbox{resp.}\quad \lim_{t\to T_-} R(g(t)) =-\infty\right).
$$
\end{proposition}

Long-time existence therefore follows for those Lie groups with the property that any left-invariant metric has non-positive
scalar curvature (see \cite{Brr}).

\begin{corollary}\label{acRF-reg2}
Any left-invariant acRF-solution $g(t)$ on a unimodular Lie group covered by the euclidean space is immortal (i.e. $T_+=\infty$).  In particular, this holds for unimodular solvable Lie groups, and consequently for invariant acRF-solutions on any compact solvmanifold.
\end{corollary}

Recall that a Lie group is covered by the euclidean space if and only if the semisimple part of its Lie algebra is a sum of ideals all isomorphic to $\slg_2(\RR)$.

In the nilpotent case, the acRF has been studied in \cite{minimal}, where besides the uniqueness of acRF-solitons up to isometry and scaling, the following characterizations were given:

\begin{itemize}
\item $g$ is an acRF-soliton.

\item $g$ minimizes the functional $|\Ricac|/|R|$ among all left-invariant metrics on the given Lie group.

\item $\Ricac(g)=cI+D$ for some $c\in\RR$, $D\in\Der(\ggo)$.
\end{itemize}

In \cite{Frn}, left-invariant acRF-solitons on all symplectic structures on $2$ and $3$-step $6$-dimensional nilpotent Lie groups have been classified.

Concerning regularity and convergence, the following results can be proved in much the same way as in \cite{nilricciflow}, where the Ricci flow on nilmanifolds was studied (replace reference [19] by \cite{minimal} to get (6.1) and Theorem 6.2 for $\Ricac$ in \cite{nilricciflow}).  Let $g(t)$ be a left-invariant acRF-solution on a symplectic nilpotent Lie group $(N,\omega)$ of dimension $2n$, then the following holds:

\begin{itemize}
\item $g(t)$ is defined for $t\in[0,\infty)$ and there exists a constant $C_n$ that only depends on $n$ such that
$|\Riem(g(t))|\leq C_n/t$ for all $t\in(0,\infty)$; in particular, $g(t)$ is a type-III solution.

\item The functional $|\Ricac|/|R|$ is strictly decreasing along $g(t)$ for all $t$, unless $g_0$ is an acRF-soliton.

\item The metrics $g(t)$ converge in the pointed sense to a flat metric, as $t\to\infty$.

\item The bracket flow $\ddt\mu=\delta_\mu(\Ricci_\mu^{ac})$ is precisely the negative gradient flow of the functional $|\Ricci^{ac}|^2$, and consequently, $\mu/|\mu|$ always converges to a unique Lie bracket $\lambda$ (no chaos).  Moreover, it follows from \cite[Theorem 6.4]{Jbl} that $\lambda$ is isomorphic to $\mu_0$ if and only if $G_{\mu_0}$ admits an acRF-soliton.

\item If $g_0$ is nonflat, then the metrics $|\! R(g(t)) \!| g(t)$ converges in the pointed sense to a nonflat acRF-soliton $g_\infty$, as $t\to\infty$.  The metric $g_\infty$ is isometric to a left-invariant metric on some nilpotent Lie group $\widetilde{N}$, though possibly non-isomorphic to $N$.
\end{itemize}

\subsection{Examples}
In order to illustrate many of the aspects of the approach proposed in the paper, we next give two simple examples in detail for SCF.

\begin{example}\label{n4}
Let $\mu=\mu_{a,b}$ be the nilpotent Lie bracket on $\ggo=\RR^4$ defined by
$$
\mu(e_1,e_2)=ae_3, \qquad \mu(e_1,e_3)=be_4.
$$
Consider $g_0$, the canonical inner product, and
$$
\omega_0=e^1\wedge e^4+e^2\wedge e^3, \qquad
J_0=\left[\begin{smallmatrix} &&0&-1\\ &&-1&0\\ 0&1&&\\ 1&0&&
\end{smallmatrix}\right],
$$
where $\{ e_i\}$ denotes the canonical basis of $\RR^4$ and $\{ e^i\}$ its dual basis.

It is easy to check that $(\mu,\omega_0,g_0)$ is an almost-K\"ahler structure, whith Chern-Ricci form is given by $p=-\tfrac{ab}{2}e^1\wedge e^2$, and
$$
P=\unm\left[\begin{smallmatrix} &&0&0\\ &&0&0\\ ab&0&&\\ 0&-ab&&
\end{smallmatrix}\right], \qquad
\Ricac=-\unc\left[\begin{smallmatrix} a^2+2b^2&&&\\ &2a^2-b^2&&\\ &&-2a^2+b^2&\\ &&&-a^2-2b^2
\end{smallmatrix}\right].
$$
If exactly one of $a,b$ vanishes, then $(G_\mu,\omega_0,g_0)$ is the simply connected cover of the Kodaira-Thurston manifold.  This Lie group, which is isomorphic to $\RR\times H_3$, where $H_3$ denotes the Heisenberg group, admits a unique almost-K\"ahler structure up to equivalence and scaling, and so such structure is necessarily a $(p,q)$-soliton for any $(p,q)$-flow.  This can also be checked by using \eqref{SCF-algsol}.

We can therefore assume that $a,b>0$, up to equivalence.  By using that the space of derivations of $\mu$ is given by
$$
\Der(\mu)=\left\{\left[\begin{smallmatrix} \alpha&0&0&0\\ \ast&\beta&0&0\\ \ast&\gamma&\alpha+\beta&0\\ \ast&\ast&b\gamma/a&2\alpha+\beta
\end{smallmatrix}\right]:\alpha,\beta,\gamma\in\RR\right\},
$$
it is not hard to show that $\mu$ is a SCF-soliton if and only if $a=b$.  In that case, condition \eqref{SCF-algsol} holds:
$$
P+\Ricac=-\tfrac{5}{4}a^2I+\unm a^2 \left[\begin{smallmatrix} 1&&&\\ 0&2&&\\ 1&0&3&\\ 0&-1&0&4
\end{smallmatrix}\right] \in\RR I+\Der(\mu).
$$
Note that $P$ is always a derivation, and so actually condition \eqref{PQac} does hold.

It is straightforward to see that this family is invariant under the bracket flow, which is equivalent to the following ODE system for the variables $a(t),b(t)$:
$$
\left\{\begin{array}{l}
a'= -\tfrac{5}{4}a^3, \\ \\
b'= -\tfrac{5}{4}b^3.
\end{array}\right.
$$
This can be explicitly integrated as
$$
a(t)=\left(\tfrac{5}{2}t+\tfrac{1}{a_0^2}\right)^{-1/2}, \qquad  b(t)=\left(\tfrac{5}{2}t+\tfrac{1}{b_0^2}\right)^{-1/2},
$$
and thus all these SCF-solutions are immortal (i.e. $T_+=\infty$) and $\mu(t)\to 0$, as $t\to\infty$.  It follows from Section \ref{conver} that the original SCF-solution $(G_{\mu_0},\omega(t),g(t))$ converges to the K\"ahler euclidean space $(\RR^4,\omega_0,g_0)$ in the pointed sense, as $t\to\infty$.  Moreover, since
$$
\lim\limits_{t\to\infty} \mu(t)/|\mu(t)|=\lambda:=\mu_{1/2,1/2}, \qquad\forall a_0,b_0>0,
$$
we obtain pointed convergence of the (normalized) SCF-solution $(G_{\mu_0},c(t)\omega(t),c(t)g(t))$ toward the SCF-soliton $(G_\lambda,\omega_0,g_0)$, where  $c(t)=|\mu(t)|^2$.
\end{example}

The SCF-evolution of the structures in Example \ref{n4} have already been studied in \cite[Sections 3.1, 3.3]{Pk}, in the standard way.

\begin{example}\label{anna}
Let $\mu=\mu_{a,b}$ be the (non-unimodular) solvable Lie bracket on $\ggo=\RR^4$ defined by
$$
\mu(e_1,e_2)=-ae_2, \qquad \mu(e_1,e_3)=2ae_3, \qquad \mu(e_1,e_4)=ae_4, \qquad \mu(e_2,e_3)=be_4,
$$
and consider $g_0$ the canonical inner product,
$$
\omega_0=e^1\wedge e^3+e^2\wedge e^4, \qquad
J_0=\left[\begin{smallmatrix} &&-1&0\\ &&0&-1\\ 1&0&&\\ 0&1&&
\end{smallmatrix}\right]
$$
The almost-K\"ahler structure $(\mu,\omega_0,g_0)$ has $p=-a(4a+b) e^1\wedge e^3$, and
$$
P=-\unm a(4a+b)\left[\begin{smallmatrix} 1&&&\\ &0&&\\ &&1&\\ &&&0
\end{smallmatrix}\right], \qquad
\Ricac=-\unc(4a^2-b^2)\left[\begin{smallmatrix} 1&&&\\ &-2&&\\ &&-1&\\ &&&2
\end{smallmatrix}\right].
$$
For any $a,b\ne 0$, $\mu$ is isomorphic to the Lie algebra $\dg_{4,2}$, as denoted in \cite[Proposition 2.1]{Ovn}, and if in addition $b=2a$, then $(\mu_{a,2a},\omega_0,g_0)$ is equivalent (up to scaling) to the unique left-invariant strictly almost-K\"ahler structure on a $4$-dimensional Lie group having $\Ricac=0$ found in \cite{Fn}.  It is easy to see that this is the unique SCF-soliton among this family for $a,b\ne 0$. Note that actually condition \eqref{PQac} does hold:
$$
P=-3a^2I+3a^2\left[\begin{smallmatrix} 0&&&\\ &1&&\\ &&0&\\ &&&1
\end{smallmatrix}\right] \in\RR I+\Der(\mu_{a,2a}).
$$
It is straightforward to show that this family is invariant under the bracket flow, which is equivalent to the following ODE system for the variables $a(t),b(t)$:
$$
\left\{\begin{array}{l}
a'= (-9a^2+\unc b^2-2ab)a, \\ \\
b'= (-3a^2-\tfrac{5}{4}b^2-2ab)b.
\end{array}\right.
$$
By a standard qualitative analysis, we obtain as in the above example long-time existence for all these SCF-solutions and that $(G_{\mu_0},\omega(t),g(t))$ converges to $(\RR^4,\omega_0,g_0)$ in the pointed sense, as $t\to\infty$.  Furthermore,
$$
\lim\limits_{t\to\infty} \mu(t)/|\mu(t)|=\lambda:=\mu_{1,2}/\sqrt{20}, \qquad\forall a_0,b_0>0,
$$
and thus pointed convergence of $(G,c(t)\omega(t),c(t)g(t))$ toward the SCF-soliton $(G_\lambda,\omega_0,g_0)$ follows for $c(t)=|\mu(t)|^2$.
\end{example}

\section{Appendix: Chern-Ricci form}\label{chern}

Let $(M,\omega,g,J)$ be a $2n$-dimensional almost-hermitian manifold.  The {\it Chern connection} is the unique connection $\nabla$ on $M$ which is hermitian (i.e. $\nabla\omega=0$, $\nabla g=0$, $\nabla J=0$) and its torsion satisfies $T^{1,1}=0$.  In terms of the Levi Civita connection $D$ of $g$, the Chern connection is given by
$$
g(\nabla_XY,Z)=g(D_XY,Z)-\unm d\omega(JX,Y,Z) - g(X,N(Y,Z)),
$$
where $N(X,Y)=[JX,JY]-[X,Y]-J[JX,Y]-J[X,JY]$ is the Nijenhuis tensor (see e.g. \cite[(2.1)]{Vzz2}, \cite[(2.1)]{DsCVzz} and \cite[Section 2]{TstWnk} for different equivalent descriptions).  We note that $\nabla=D$ if and only if $(M,\omega,g,J)$ is K\"ahler.  In the almost-K\"ahler case, the above formula reduces to
$$
\nabla_XY=D_XY+\unm(D_XJ)JY.
$$
The {\it Chern-Ricci form} $p=p(\omega,g)$ is defined by
$$
p(X,Y)=\sum_{i=1}^{n} g(R(X,Y)e_i,Je_i) = \sqrt{-1} \sum_{i=1}^{n} g(R(X,Y)Z_i,Z_{\overline{i}}),
$$
where $R(X,Y)=\nabla_{[X,Y]} - [\nabla_X,\nabla_Y]$ is the curvature tensor of $\nabla$ and
$$
\{ e_1,\dots,e_n,Je_1,\dots,Je_n\}
$$
is a local orthonormal frame for $g$ with corresponding local unitary frame
$$
Z_i:=(e_i-\sqrt{-1} Je_i)/\sqrt{2}, \qquad Z_{\overline{i}}:=(e_i+\sqrt{-1} Je_i)/\sqrt{2}.
$$
The Chern-Ricci form is always closed and locally $p=d\theta$, where $\theta$ is the $1$-form given by
\begin{align*}
\theta(X)= & \sqrt{-1} \sum_{i=1}^n g(\nabla_XZ_i,Z_{\overline{i}}) \\
=& -\tfrac{1}{2}\sum_{i=1}^n g([X,e_i],Je_i) -g([X,Je_i],e_i) +g([JX,e_i],e_i) +g([JX,Je_i],e_i).
\end{align*}
(See \cite[(3.3)]{Vzz2}).  If $J$ is integrable, then $p$ is a $(1,1)$-form (see e.g. \cite[Section 2]{TstWnk}), and in the K\"ahler case $p$ equals the Ricci form $\ricci(J\cdot,\cdot)$.

By Chern-Weil theory $[p]=2\pi c_1(M,J)$, where $c_1(M,J)\in H^2(M,\RR)$ is the first Chern class (see e.g. \cite[Chapter 16]{Mrn} or \cite[Chapter 12]{KbyNmz} for further information).

The Chern-Ricci form $p$ of a left-invariant almost-hermitian structure $(\omega,g,J)$ on a Lie group with Lie algebra $\ggo$ is given by
\begin{equation}\label{CRform}
p(X,Y)=-\unm\tr{J\ad{[X,Y]}} + \unm\tr{\ad{J[X,Y]}}, \qquad\forall X,Y\in\ggo.
\end{equation}
(See \cite[Proposition 4.1]{Vzz2} or \cite{Pk}).  Remarkably, $p$ only depends on $J$.  If $P\in\End(\ggo)$ is the Chern-Ricci operator, i.e. $p=\omega(P\cdot,\cdot)$, then by \eqref{CRform} $P$ vanishes on the center of $\ggo$.

It follows from \cite[Proposition 4.2]{Vzz2} that $p$ vanishes under any of the following conditions:

\begin{itemize}
\item $J$ bi-invariant (i.e. $[J\cdot,\cdot]=J[\cdot,\cdot]$).

\item $J$ anti-bi-invariant (i.e. $[J\cdot,\cdot]=-J[\cdot,\cdot]$).

\item $J$ abelian (i.e. $[J\cdot,J\cdot]=[\cdot,\cdot]$) and $\ggo$ unimodular.
\end{itemize}

It is proved in \cite{CRF} that if $J$ is integrable and $\ggo$ is nilpotent, then $p=0$.

On the other hand, in the case when $\omega$ is closed, it is proved in \cite{Frn} that
$$
P=\ad{Z}+(\ad{Z})^{t_\omega} =\ad{Z}+J(\ad{Z})^tJ^{-1},
$$
where $Z\in\ggo$ is defined by $p(X,Y)=\omega(Z,[X,Y])$ for all $X,Y\in\ggo$, and furthermore, that $P$ is a nilpotent operator if $\ggo$ is unimodular.

\end{document}